\documentclass[11pt, intlimits]{amsart}
\usepackage{amsmath,amssymb}
\usepackage{verbatim}

\begin{document}

\newtheorem{thm}{Theorem}[section]
\newtheorem{lem}[thm]{Lemma}
\newtheorem{prop}[thm]{Proposition}
\newtheorem{cor}[thm]{Corollary}
\newtheorem{defn}[thm]{Definition}
\newtheorem*{remark}{Remark}
\newtheorem{conj}[thm]{Conjecture}

\numberwithin{equation}{section}

\newcommand{\Z}{{\mathbb Z}} %cph changed from \mathbf
\newcommand{\Q}{{\mathbb Q}}
\newcommand{\R}{{\mathbb R}}
\newcommand{\C}{{\mathbb C}}
\newcommand{\N}{{\mathbb N}}
\newcommand{\FF}{{\mathbb F}}
\newcommand{\fq}{\mathbb{F}_q}
\newcommand{\rmk}[1]{\footnote{{\bf Comment:} #1}}

\renewcommand{\mod}{\;\operatorname{mod}}
\newcommand{\ord}{\operatorname{ord}}
\newcommand{\TT}{\mathbb{T}}
\renewcommand{\i}{{\mathrm{i}}}
\renewcommand{\d}{{\mathrm{d}}}
\renewcommand{\^}{\widehat}
\newcommand{\HH}{\mathbb H}
\newcommand{\Vol}{\operatorname{vol}}
\newcommand{\area}{\operatorname{area}}
\newcommand{\tr}{\operatorname{tr}}
\newcommand{\norm}{\mathcal N} % norm =(\frac{ n+\sqrt{n^2-4}} 2)^2
\newcommand{\intinf}{\int_{-\infty}^\infty}
\newcommand{\ave}[1]{\left\langle#1\right\rangle} %  average
\newcommand{\Var}{\operatorname{Var}}
\newcommand{\Prob}{\operatorname{Prob}}
\newcommand{\sym}{\operatorname{Sym}}
\newcommand{\disc}{\operatorname{disc}}
\newcommand{\CA}{{\mathcal C}_A}
\newcommand{\cond}{\operatorname{cond}} % conductor
\newcommand{\lcm}{\operatorname{lcm}}
\newcommand{\Kl}{\operatorname{Kl}} %Kloosterman sum
\newcommand{\leg}[2]{\left( \frac{#1}{#2} \right)}  % Legendre symbol
\newcommand{\Li}{\operatorname{Li}}

\newcommand{\sumstar}{\sideset \and^{*} \to \sum}

\newcommand{\LL}{\mathcal L} %L-function of u
\newcommand{\sumf}{\sum^\flat}
\newcommand{\Hgev}{\mathcal H_{2g+2,q}}
\newcommand{\USp}{\operatorname{USp}}
\newcommand{\conv}{*}
\newcommand{\dist} {\operatorname{dist}}
\newcommand{\CF}{c_0} % Fejer constant
\newcommand{\kerp}{\mathcal K}

\newcommand{\Cov}{\operatorname{cov}}
\newcommand{\Sym}{\operatorname{Sym}}

\newcommand{\ES}{\mathcal S} % sums over AP
\newcommand{\EN}{\mathcal N} % sum over short intervals
\newcommand{\EM}{\mathcal M} % sum over pols of deg n
\newcommand{\Sc}{\operatorname{Sc}} %Secular coefficients
\newcommand{\Ht}{\operatorname{Ht}}

\newcommand{\E}{\operatorname{E}} % expectation
\newcommand{\sign}{\operatorname{sign}} %sign

\newcommand{\divid}{d} % the divisor function
\newcommand{\inv}{\theta}

\newcommand{\Res}{\operatornamewithlimits{Res}} %residue

\title[Squarefree polynomials and M\"obius values]
{Squarefree polynomials and M\"obius values in short intervals and arithmetic progressions  }

\author{J.P. Keating  and Z. Rudnick}

\address{School of Mathematics, University of Bristol, Bristol BS8 1TW, UK}
\email{j.p.keating@bristol.ac.uk}

\address{Raymond and Beverly Sackler School of Mathematical Sciences,
Tel Aviv University, Tel Aviv 69978, Israel}
\email{rudnick@post.tau.ac.il}
\date{\today}
\thanks{ JPK gratefully acknowledges support under EPSRC Programme Grant EP/K034383/1
LMF: L-Functions and Modular Forms, a grant from Leverhulme Trust, a Royal Society Wolfson Merit Award, a Royal Society Leverhulme Senior Research Fellowship, and by the Air Force Office of Scientific Research, Air Force Material Command, USAF, under grant number FA8655-10-1-3088.  ZR is similarly grateful for support from the Friends of the Institute for Advanced Study, 
the European Research Council under the European Union's Seventh
Framework Programme (FP7/2007-2013) / ERC grant agreement
n$^{\text{o}}$ 320755, and from the Israel Science Foundation
(grant No. 925/14). }


\begin{abstract}
We calculate the mean and variance of sums of the M\"obius function $\mu$ and the
indicator function of the squarefrees $\mu^2$, in both short intervals and arithmetic progressions, in the context of
the ring $\fq[t]$ of polynomials over a finite field $\fq$ of $q$
elements, in the limit $q\to \infty$.  We do this by relating the sums in question to certain matrix integrals over the unitary group, using recent equidistribution results due to N. Katz, and then by evaluating these integrals.  In many cases our results mirror what is either known or conjectured for the corresponding problems involving sums over the integers, which have a long history.  In some cases there are subtle and surprising differences.  The ranges over which our results hold is significantly greater than those established for the corresponding problems in the number field setting.
\end{abstract}
\maketitle

\tableofcontents

\section{Introduction}

% A fundamental feature of integer
%arithmetic is that the prime factorization of an integer does not
%depend continuously on its location. This feature manifests itself
%in wild oscillations of the values of various naturally defined
%arithmetic functions. Nonetheless it is expected that when averaged
%over intervals which are not too short, several arithmetic functions
%have smoother behaviour and a dominant theme in number theory has
%been to try to understand the statistics of such averages over
%"short" intervals and over arithmetic progressions. 
% \marginpar{This needs to be rewritten ! Emphasize clearly the insights gained for ordinary number theory}
The goal of this paper is to investigate the
fluctuation of sums  of two important arithmetic
functions, the M\"obius function $\mu$ and the
indicator function of the squarefrees $\mu^2$, in the context of the
the ring $\fq[t]$ of polynomials over a finite field $\fq$ of $q$
elements, in the limit $q\to \infty$. 
The problems we address, which concern sums over short intervals and arithmetic progressions, mirror long-standing questions over the integers,
where they are largely unknown. In our setting we succeed in
giving definitive answers.  

Our approach differs from those traditionally employed in the number field setting: we use recent equidistribution results due to N. Katz, valid in the large-$q$ limit, to express the mean and variance of the fluctuations in terms of matrix integrals over the unitary group.  Evaluating these integrals leads to explicit formulae and precise ranges of validity.  For many of the problems we study, the formulae we obtain match the corresponding number-field results and conjectures exactly, providing further support in the latter case.  However, the ranges of validity that we can establish are significantly greater than those known or previously conjectured for the integers, and we see our results as supporting extensions to much wider ranges of validity in the integer setting.  Interestingly, in some other problems we uncover subtle and surprising differences between the function-field and number-field asymptotics, which we examine in detail. 

We now set out our main results in a way that enables comparison with the corresponding problems for the integers.

%\section{Examples} \marginpar{change}
%We present several examples of arithmetic functions and the results
%for the variance. All of these problems have a long history in the
%number field case, described separately below.

\subsection{The M\"obius function}

%\subsubsection{Sums of the M\"obius function in short intervals}

It is a standard heuristic to assume that the M\"obius function
behaves like a random $\pm 1$ supported on the squarefree integers,
which have density $1/\zeta(2)$ (see e.g. \cite{CS}). Proving anything in this direction
is not easy. Even demonstrating cancellation in the sum
$M(x):=\sum_{1\leq n\leq x}\mu(n)$, that is that $M(x)=o(x)$, is
equivalent to the Prime Number Theorem. The Riemann Hypothesis is
equivalent to square-root cancellation: $M(x) =
O(x^{1/2+o(1)})$.

For sums of $\mu(n)$ in blocks of length $H$,
\begin{equation}
M(x;H):=\sum_{|n-x|<H/2} \mu(n)
\end{equation}
%where  $X^\epsilon<H=H(X)<X^{1-\epsilon}$, 
it has been shown that
there is cancellation for $H\gg x^{7/12+o(1)}$ \cite{Motohashi
Mobius, Ramachandra}, and assuming the Riemann Hypothesis one can
take $H\gg x^{1/2+o(1)}$.   If one wants cancellation only for ``almost all" values of $x$, then more is known. In particular, very recently M\"atomaki and Radziwi\l{}\l{}  \cite{MR} have shown (unconditionally) that 
$$\frac 1X\int_X^{2X} M(x;H)^2dx = o(H^2)$$
whenever $H=H(X)\to \infty$ as $X\to\infty$, and in particular $M(x;H) = o(H)$ for almost all $x\in [X,2X]$. 
 
We expect the normalized sums $M(x;H)/\sqrt{H}$ to have mean zero
(this follows from the Riemann Hypothesis) and variance
$6/\pi^2=1/\zeta(2)$: 
\begin{equation}\label{GC conj}
\frac 1X\int_X^{2X} |M(x;H)|^2 \sim \frac H{\zeta(2)}.
\end{equation}
Moreover,  $M(x;H)/\sqrt{H/\zeta(2)}$ is believed to have
a normal distribution asymptotically.
These conjectures were formulated and investigated
numerically by Good and Churchhouse \cite{GC} in 1968, and further
studied by Ng  \cite{Ng},
%(see http://www.cs.uleth.ca/~nathanng/RESEARCH/mobiusshort.pdf)
 who carried out an analysis using the Generalized Riemann Hypothesis (GRH)
 and a strong version of Chowla's conjecture
 on correlations of M\"obius, showing that \eqref{GC conj} is valid
for $H\ll X^{1/4-o(1)}$ and that Gaussian distribution holds
 (assuming these conjectures) for $H\ll X^\epsilon$. It is important
 that the length $H$ of the interval be significantly smaller than
 its location, that is $H<X^{1-\epsilon}$, since otherwise one
 expects non-Gaussian statistics, see \cite{Ng2004}.

%\subsubsection{The M\"obius function in arithmetic progressions}

 Concerning arithmetic
progressions, Hooley \cite{Hooley BDHIII} studied the following
averaged form of the total variance (averaged over moduli)
\begin{equation}
V(X,Q):=\sum_{Q'\leq Q}\sum_{A\bmod Q'} \Big( \sum_{\substack{n\leq
X\\ n=A\bmod Q'}} \mu(n)\Big)^2
\end{equation}
and showed that for $Q\leq X$,
\begin{equation}
V(X,Q) = \frac{6QX}{\pi^2} +O(X^2(\log X)^{-C})
\end{equation}
for all $C>0$, which yields  an asymptotic result for $X/(\log X)^C
\ll Q<X$.

For polynomials over a finite field $\fq$, the M\"obius function is
defined as for the integers, namely by $\mu(f)=(-1)^k$ if
$f$ is a scalar multiple of a product of $k$ distinct monic
irreducibles, and $\mu(f)=0$ if $f$ is not squarefree. The analogue
of the full sum $M(x)$ is the sum over all monic polynomials $\EM_n$
of given degree $n$, for which we have
\begin{equation}
%\EN(n):=
\sum_{f\in \EM_n}\mu(f) =
\begin{cases} 1,&n=0\\ -q,&n=1\\0,&n\geq 2
\end{cases}
\end{equation}
so that in particular the issue of size is trivial\footnote{This
ceases to be the case when dealing with function fields of higher
genus, see e.g. \cite{Cha, Humphries}}. However that is no longer
the case when considering sums over ``short intervals", that is over
sets of the form
\begin{equation}
I(A;h) = \{f:||f-A||\leq q^h\}
\end{equation}
where  $A\in \EM_n$ has degree $n$, $0\leq h\leq n-2$
and\footnote{For $h=n-1$, $I(A;n-1) = \EM_n$ is the set of all monic
polynomials of degree $n$} the norm is
\begin{equation}
||f||:=\#\fq[t]/(f) = q^{\deg f}\;.
\end{equation}
To facilitate comparison between statements for number field results
and for function fields, we use a rough dictionary
\begin{equation}\label{dictionary}
\begin{split}
X\leftrightarrow q^n , & \quad\log X\leftrightarrow n\\
H\leftrightarrow q^{h+1}, & \quad \log H\leftrightarrow h+1
\end{split}
\end{equation}

  Set
\begin{equation}
\EN_\mu(A;h):= \sum_ {f\in I(A;h)} \mu(f)\;.
\end{equation}
The number of summands here is $q^{h+1}=:H$ and we want to display
cancellation in this sum and study its statistics as we vary the
``center"  $A$ of the interval.

We can demonstrate cancellation in the short interval sums
$\EN_\mu(A;h)$ in the large finite field limit $q\to \infty$, $n$
fixed (we assume $q$ is odd throughout the paper):
\begin{thm}\label{prop asymp mobius short}
If   $2\leq h\leq n-2$  then for all $A$ of degree $n$,
\begin{equation*}
\Big| \EN_\mu(A;h) \Big| \ll_{n} \frac H{\sqrt{q}}
\end{equation*}
the implied constant uniform in $A$, depending only on $n= \deg A$.
\end{thm}
For $h=0,1$ this is no longer valid, that is there are $A$'s for
which there is no cancellation, see \S~\ref{sec:asymp}.

We next investigate the statistics of $\EN_\mu(A;h)$ as $A$ varies
over all monic polynomials of given degree $n$, and $q\to \infty$.
 It is easy to see that for $n\geq 2$, the mean value of $\EN_\mu(A;h)$ is
 $0$.
 Our main result concerns
the variance:   %(again $H= q^{h+1}$)
\begin{thm}\label{thm mobius}
If  $0\leq h\leq n-5$ then as $q\to \infty$, $q$ odd,
\begin{equation*}
\Var\EN_\mu(\bullet;h)= \frac 1{q^n}\sum_{A\in \EM_n}
|\EN_\mu(A;h)|^2 \sim H\int_{U(n-h-2)}| \tr \Sym^n U|^2 dU = H
\end{equation*}
\end{thm}
 This is consistent with the Good-Churchhouse conjecture \eqref{GC conj} if we
write it as $H/\zeta_q(2)$, where
$$\zeta_q(s) = \sum_{f \, monic} \frac 1{||f||^s},\quad {\rm Re}(s)>1,$$
which tends to $1$ as $q\to \infty$,  and $H=q^{h+1}$ is the number
of monic polynomials in the short interval.

A version of Theorem~\ref{thm mobius} valid for $h<n/2$ (``very short" intervals) has recently been obtained by Bae, Cha and Jung \cite{BCJ} using the method of our earlier paper \cite{KR}.

Analogous results can be obtained for sums over arithmetic
progressions, see \S~\ref{Sec:mob AP}.

% A problem which we have not been able to approach is the behaviour
% of higher moments and the limiting value distribution of
% $\EN_\mu(A;h)$, which we believe to be Gaussian when $q\to \infty$.
% marginpar{keep this?} 

\subsection{Squarefrees}

%\section{squarefrees in short intervals}

%\subsection{Over the integers}

It is well known that the density of squarefree integers is
$1/\zeta(2) = 6/\pi^2$, and an elementary sieve shows
\begin{equation}\label{elt Q}
Q(x) :=\#\{n\leq x: n \mbox{ squarefree}\} = \frac{x}{\zeta(2)}  +
O(x^{1/2}) \;.
\end{equation}
No better exponent is known for the remainder term.
%(and this is related to the Riemann Hypothesis).
Using zero-free regions for
$\zeta(s)$, Walfisz gave a remainder term of the form
$x^{1/2}\exp(-c(\log x)^{3/5+o(1)})$.
 Assuming RH, the exponent $1/2$ has been improved \cite{Axer, MV,  BP}, currently to
 $17/54=0.31$ \cite{Jia}.
It is expected that
 \begin{equation}
   Q(x) = \frac{x}{\zeta(2)} + O(x^{1/4+o(1)}) \;.
 \end{equation}

Since the density is known, we wish to understand to what extent we
can guarantee the existence of squarefrees in short intervals
$(x,x+H]$; moreover, when do we still expect to have an asymptotic
formula for the number
\begin{equation}
Q(x,H):=%\#\{N-\frac 12 H<n\leq n+\frac 12 H: n \mbox{ squarefree}\}=
\sum_{|n-x|\leq  \frac H2} \mu^2(n) =Q(x+H)-Q(x)
\end{equation}
 of squarefrees in the interval $(x,x+H]$; that is when do we still
have
\begin{equation}\label{asym H}
Q(x;H) \sim \frac{H}{\zeta(2)} \;.
\end{equation}
In view of the bound of $O(x^{1/2})$ for the remainder term  in
\eqref{elt Q}, this holds for $H\gg x^{1/2+o(1)}$. However, one can
do better without improving on the remainder term in \eqref{elt Q}.
This was first done by Roth \cite{Roth} who by an elementary method
showed that the asymptotic \eqref{asym H} persists for $H\gg
x^{1/3+o(1)}$. Following improvements by Roth himself (exponent
$3/13$) and  Richert \cite{Richert} in 1954 (exponenent $2/9$), the
current best bound is by  Tolev \cite{Tolev} (2006) (building on
earlier work by Filaseta and Trifonov) who gave $H\gg x^{1/5+o(1)}$.
It is believed that \eqref{asym H} should hold for $H\gg
x^{\epsilon}$ for any $\epsilon>0$, though there are intervals of
size $H\gg \log x/\log\log x$ which contain no squarefrees, see
\cite{Erdos 1951}.

% There have  been several studies on counting squarefree integers in short
%intervals. If we set
%\begin{equation}
%Q(N,H):=%\#\{N-\frac 12 H<n\leq n+\frac 12 H: n \mbox{ squarefree}\}=
%\sum_{|n-N|\leq  \frac H2} \mu^2(n)\;,
%\end{equation}
%then one wants  to get an asymptotic formula of the form
%\begin{equation}
%Q(N,H)\sim \frac 1{\zeta(2)} H
%\end{equation}
% for $H=o(N)$, either individually or for almost all such intervals.
%Tolev \cite{Tolev} (building on earlier work by Filaseta and
%Trifonov) showed that an asymptotic formula holds individually
%provided $H/N^{1/5} \log N \to \infty$ (for earlier work, see Roth
%\cite{Roth},  Richert \cite{Richert}, and Filaseta \cite{Filaseta}
%among others).

As for almost-everywhere results, one way to proceed goes through a
study of the variance of $Q(x,H)$. In this direction, Hall
\cite{Hall} showed that provided  $H=O(x^{2/9-o(1)})$, the variance
of $Q(x,H)$ admits an asymptotic formula:
\begin{equation}\label{eq:Hall thm}
\frac 1x \sum_{n\leq x} \left |Q(n,H) - \frac H{\zeta(2)}\right |^2
\sim A\sqrt{H}\;,
\end{equation}
with
\begin{equation}\label{Hall constant int}
A = \frac{\zeta(3/2)}\pi \prod_p(\frac{p^3-3p+2 }{p^3})\;.
\end{equation}
Based on our results below, we expect this asymptotic formula to
hold for $H$ as large as $x^{1-\epsilon}$.

%\section{squarefrees  in arithmetic progressions}

%\subsection{Over the integers}

Concerning arithmetic progressions, denote by
$$S(x;Q,A) =\sum_{\substack{n\leq x\\n=A\bmod Q}}\mu(n)^2$$
the number of squarefree integers in the arithmetic progression
$n=A\bmod Q$.
%Then for fixed $Q$,
%Prachar \cite{Prachar} showed
%\begin{equation}
%  M(x;Q,A) \sim f(Q;A) x
%\end{equation}
%for certain constants $f(Q,A)$;
%at this point it suffices to say
%that $f(Q,A)\neq 0$ if and only if $\gcd(A,Q)$ is squarefree.
%if $(\gcd(A,Q)=1$ then $f(Q;A) = \frac 1{\zeta(2)} \frac 1Q \prod_{p\mid Q}(1-\frac 1{p^2})^{-1}$.
Prachar \cite{Prachar}  showed that for $Q<x^{2/3-\epsilon}$, and
$A$ coprime to $Q$,
%gave a remainder term (when $\gcd(A,Q)=1$?)
%\marginpar{there is  no apriori reason to assume (A,Q)=1 here}
\begin{equation}
   S(x;Q,A)\sim \frac 1{\zeta(2)} \prod_{p\mid Q} (1-\frac 1{p^2})^{-1} \frac
   xQ =\frac 1{\zeta(2)} \prod_{p\mid Q} (1+\frac 1{p})^{-1} \frac
   x{\phi(Q)}
%+ O\left( (\frac xQ)^{1/2} + Q^{1/2+o(1)}   \right)
\end{equation}

In order to understand the size of the remainder term,
%in arithmetic progressions,
one studies  the  variance
\begin{equation}
  \Var(S ) = \frac 1{\phi(Q)} \sum_{\gcd(A,Q)=1} \left| S(x;Q,A)-  \frac 1{\zeta(2)} \prod_{p\mid Q} (1-\frac 1{p^2})^{-1} \frac xQ  \right|^2
\end{equation}
as well as a version   where the sum is over all residue classes
$A\bmod Q$, not necessarily coprime to $Q$, and the further averaged
form over all moduli $Q'\leq Q$ a-la Barban, Davenport \& Halberstam,
see  \cite{Warlimont1980}.
%\begin{equation}
%  V'(x;Q) = \sum_{Q\,' \leq Q} G'(x,Q\,')
%\end{equation}
%Following prior work on the subject%\cite{Orr, Croft}
%, Warlimont \cite{Warlimont1980} showed that for $1\leq Q\leq x$,
%\begin{equation}
%  V'(x;Q) = c_2x^{1/2}Q^{3/2} + O(x^{1/4}Q^{7/4}) + O(x^{3/2+o(1)})
%\end{equation}
%which gives an asymptotic formula provided $Q>x^{2/3+o(1)}$. See
%also Vaughan's paper  \cite{Vaughan} of 2005.

Without averaging over moduli, Blomer \cite[Theorem 1.3]{Blomer}
gave an upper bound for the variance,
%showed
%\begin{equation}
%G( x,Q) \ll x^\epsilon \left\{ x+ \min(\frac{x^{5/3}}Q+ Q^2))
%\right\}
%\end{equation}
%Note: Blomer summed only over residue classes co-prime to $Q$.
which was very recently improved by Nunes \cite{Nunes}, who gave an
asymptotic for the variance in the range $X^{\frac
{31}{41}+\epsilon}<Q<X^{1-\epsilon}$
\begin{equation}\label{Nunes}
  \Var(S )
%\frac 1{\phi(Q)} \sum_{\gcd(A,Q)=1}  \left| M(x;Q,A)-f(Q;A)x \right|^2
\sim A \prod_{p\mid Q}(1-\frac 1p)^{-1}(1+\frac 2p)^{-1} \cdot
\frac{X^{1/2}}{Q^{1/2}}
\end{equation}
where $A$ is given by \eqref{Hall constant int}.
%$$C =\frac{\zeta(3/2)}{\pi \zeta(2)} \prod_p (1+\frac 2{p^2+p})$$
It is apparently not known in what range of $Q$ to expect
\eqref{Nunes} to hold. Based on our results below, we conjecture
that \eqref{Nunes} holds down to $X^\epsilon<Q$. \label{check or
move}

%\subsection{In $\fq[t]$}
Our goal here is to study analogous problems for $\fq[t]$. The total
number of squarefree monic polynomials of degree $n>1$ is (exactly)
\begin{equation}
\sum_{f\in \mathcal M_n} \mu(f)^2 = \frac {q^n}{\zeta_q(2)}\;.
\end{equation}
The number of squarefree polynomials in the short interval $I(A;h)$
is
\begin{equation}
\EN_{\mu^2}(A;h)= \sum_{f\in I(A;h)} \mu(f)^2 \;.
\end{equation}
%\begin{verbatim}
%we impose the condition f(0)\neq 0 which is unnatural here so will
%need to dispose of it later.
%\end{verbatim}

\subsubsection{Asymptotics}
We show that for any short interval/arithmetic progression, we still
have an asymptotic count of the number of squarefrees:
\begin{thm}\label{thm asymp sf}
i) If   $\deg Q<n$ and $\gcd(A,Q)=1$ then
$$ \#\{f\in \EM_n: f=A\bmod Q\,f \; {\rm squarefree}\} = \frac{q^n}{|Q|}\left( 1+O_n(\frac 1q) \right)\;.
%\sim \frac{q^n}{\Phi(Q)}
$$

ii) If $0<h\leq n-2$ then for all $A\in \EM_n$,
$$\#\{f\in I(A;h): f \; {\rm squarefree}\} =\frac{H}{\zeta_q(2)}  + O(\frac Hq) = H+O_n(\frac Hq)\;.
$$
In both cases the implied constants depend only on $n$.
\end{thm}
Note that for $h=0$, Theorem~\ref{thm asymp sf}(ii) need not hold:
If $q=p^k$ with $p$ a fixed odd prime, $n=p$ then the short interval
$I(t^n;0)=\{t^n+b:b\in \fq\}$ has no squarefrees, since
$t^p+b=(t+b^{q/p})^p$ has multiple zeros for any $b\in \fq$.

\subsubsection{Variance}
 We are able to compute the variance, the size of which turns out to depend
on the parity of the interval-length parameter $h$ in a surprising way:

\begin{thm}\label{thm var sf in si}
  Let $0\leq h\leq n-6$. Assume $q\to \infty$ with all $q$'s coprime to $6$.

 i) If $h$ is even then
\begin{equation*}
\Var{\EN_{\mu^2}(\bullet;h)} \sim q^{\frac h2}
\int\limits_{U(n-h-2)} \left| \tr \Sym^{\frac h2+1}U \right|^2 dU
 =\frac{\sqrt{H}}{\sqrt{q}}
\end{equation*}
(the matrix integral works out to be $1$).

ii) If $h$ is odd then
\begin{equation*}
\Var{\EN_{\mu^2}(\bullet;h)} \sim
q^{\frac{h-1}2}\int\limits_{U(n-h-2) }\Big|\tr U\Big |^2  dU
\int\limits_{U(n-h-2) } \Big| \tr\Sym^{\frac{h+3}2} U'\Big |^2 dU' =
\frac{\sqrt{H}}q
\end{equation*}
(both matrix integrals equal $1$).
\end{thm}
To compare with Hall's result \eqref{eq:Hall thm}, where the
variance is of order $\sqrt{H}$, one wants to set $H=\#I(A;h) =
q^{h+1}$ and then in the limit $q\to \infty$ we   get smaller
variance - either $\sqrt{H}/q^{1/2}$ ($h$ even) or $ \sqrt{H}/q$
($h$ odd). We found this sufficiently puzzling to check the analogue
of Hall's result for the polynomial ring $\fq[t]$ for the large
degree limit of {\em fixed} $q$ and $n\to \infty$. The result,
presented in  Appendix~\ref{appendix}, is consistent with
Theorem~\ref{thm var sf in si} in that for $H<(q^n)^{\frac
29-o(1)}$, the variance is
\begin{equation*}
%\begin{split}
\Var (\EN_{\mu^2}(\bullet;h) ) \sim  \sqrt{H}
 \frac{\beta_q}{1-\frac 1{q^3}}
\begin{cases}
\frac {1+\frac 1{q^2}}{\sqrt{q}},& h\;{\rm even,}\\    \\
\frac{1+\frac 1q}{q},&h\;{\rm odd,}
\end{cases}
%+ O(\frac{H^2n}{q^{n/3}})
%\\
%& = q^h (\frac 1{\zeta(2)} -\beta  \frac{ 1+\frac 1q }{1-\frac
%1{q^3}}) + \frac{\beta}{1-\frac 1{q^3}} q^{(h+1)/2}
%\begin{cases}
%\frac {1+\frac 1{q^2}}{\sqrt{q}},& h\;{\rm even}\\    \\
%\frac{1+\frac 1q}{q},&h\;{\rm odd}
%\end{cases}
%\end{split}
   \end{equation*}
%In particular get an asymptotic result provided  $h< \frac 29
%n-\log_q n $. i.e. $H<(q^n)^{\frac 29}/n$.
so that it is of order $\sqrt{H}$ for fixed $q$.

 We also obtain a similar result for arithmetic
progressions. Let $Q\in \fq[t]$ be a squarefree polynomial of degree
$\geq 2$, and $A$ coprime to $Q$. we set
\begin{equation}
\ES(A) = \sum_{\substack{f=A\bmod Q\\ f\in \mathcal M_n}} \mu^2(f)
\;.
\end{equation}
%For simplicity, we consider here the case of prime modulus. ??????
%\marginpar{squarefree}
The expected value over such $A$ is
\begin{equation}
\ave{\ES}_Q = \frac 1{\Phi(Q)} \sum_{\substack{ f\in \mathcal M_n\\
(f,Q)=1}} \mu^2(f) \sim \frac{q^n/\zeta_q(2)}{\Phi(Q)}\sim
\frac{q^n}{|Q|}.
\end{equation}
% and the variance is
% \begin{equation}
% \Var_Q(\ES) = \frac 1{\Phi(Q)^2} \sum_{\chi\neq \chi_0}
% |\EM(n;\mu^2\chi)|^2 \;.
% \end{equation}
We will show that the variance satisfies:
\begin{thm}
Fix $N\geq 1$. For any sequence of finite fields $\fq$, with $q$ odd, %$q=p^e$, of characteristic $p>2$,
and squarefree polynomials $Q\in \fq[t]$  with $\deg Q=N+1$, as
$q\to \infty$,
$$
\Var_Q(\ES)\sim  \frac {q^{n/2}}{|Q|^{1/2} } \times
\begin{cases}
 1/\sqrt{q}, & n\neq \deg Q\bmod 2 \\
 \\   1/q,  &  n= \deg Q\bmod 2 \;.
\end{cases}
$$
\end{thm}

\subsection{General approach}
It may be helpful to give an informal sketch of the general approach we take in proving 
most of the theorems stated above.  
Short intervals are transformed into sums over special arithmetic progressions, a feature special to function fields that was used in our earlier work \cite{KR}. 
Sums involving $\mu$ and $\mu^2$ that run over all monic polynomials of a given 
degree may be evaluated in terms of a zeta function that is the function-field analogue of the Riemann zeta function.  Restricting 
to short intervals or arithmetic progressions leads to sums over Dirichlet characters involving the associated L-functions.
The L-functions in question may be written in terms of unitary matrices.  It has recently been established by N. Katz that, in the 
limit when $q\rightarrow\infty$, these matrices become equidistributed in the unitary group, in the sense that the character sums we need 
are, in the large-$q$ limit, equal to integrals over the unitary group.  Evaluating these integrals leads to the formulae appearing in our theorems.

%\newpage

\section{Asymptotics for squarefrees: Proof of Theorem~\ref{thm asymp sf}}
 We want to show that almost all polynomials in an arithmetic
progression, or in a short interval are squarefree. We recall the
statement:

i) If $\deg Q<n$ and $\gcd(A,Q)=1$ then
\begin{equation}\label{asymp arith prog}
 \#\{f\in \EM_n: f=A\bmod Q,  f \; {\rm squarefree}\}
\sim \frac{q^n}{|Q|}\sim \frac{q^n}{\Phi(Q)} \;.
\end{equation}

ii) If $0<h\leq n-2$ then
\begin{equation}\label{asymp short int}
\#\{f\in I(A;h): f \;  {\rm squarefree}\} =\frac{H}{\zeta_q(2)}  +
O(\frac Hq) = H+O(\frac Hq) \;.
\end{equation}

These follow from a general result \cite{Rudnick}:
\begin{thm}\label{Rudnick thm}
 Given a separable polynomial $F(x,t)\in \fq[x,t]$, with squarefree content then
the number of monic polynomials $a\in \EM_m$, $m>0$, for which $F(a(t),t)$
is squarefree (in $\fq[t]$) is asymptotically
$$q^m + O\Big(q^{m-1}(m\deg F+\Ht(F ))\deg F\Big) \;.
$$
\end{thm}
Here if $F(x,t) = \sum_{j=0}^{\deg F} \gamma_j (t) x^j$ with
$\gamma_j(t)\in \fq[t]$ polynomials, the content of $F$ is
$\gcd(\gamma_0,\gamma_1,\dots,)$ and the height is $\Ht(f) = \max_j
\deg \gamma_j$.

For an arithmetic progression $f=A\mod Q$, $f\in \EM_n$ monic of
degree $n$,  with $\gcd(A,Q)=1$,  $\deg A<\deg Q$, we take the corresponding
polynomial to be
$$
F(x,t) = A(t) +\frac 1{\sign Q}Q(t)x
$$
where $\sign Q\in \fq^\times$ is such that $Q(t)/\sign Q$ is monic.
Then $F(x,t)$  has degree one (in $x$),
hence is certainly separable, and has content equal to $\gcd(A,Q)=1$
so is in fact primitive. The height of $F$ is $\max(\deg Q,\deg A) =
\deg Q<n$ which is independent of $q$.

Since  $\deg A<\deg Q$,
then $f=A+aQ/\sign(Q)$ is monic of degree $n$ if and only if $a$ is monic of
degree $n-\deg Q>0$, and by Theorem~\ref{Rudnick thm} the number of
such $a$ for which $F(a(t),t)$ is squarefree is
$$q^{n-\deg Q}+O(q^{n-\deg Q-1}) = \frac {q^n}{|Q|}(1+O(\frac 1q)) \;.
$$
This proves \eqref{asymp arith prog}.

To deal with the short interval case, let $0<h\leq n-2$, and $A\in
\EM_n$ be monic of degree $n$. We want to show that the number of
polynomials $f$ in the short interval $I(A;h)$ which are squarefree
is $H + O(H/q)$ (recall $H=\#I(A;h) = q^{h+1}$).

We write
$$ I(A;h) = (A+\mathcal P_{\leq h-1})\cup \coprod_{c\in \fq^\times}  
(A+c\EM_h)\;.
$$
The number of squarefrees in $A+\mathcal P_{\leq h-1}$ is at most $\#\mathcal P_{\leq h-1}= q^h$.
The squarefrees in $A+c\EM_h$ are the squarefree values at monic
polynomials of degree $h$ of the polynomial $F(x,t)=A(t)+cx$, which
has degree $1$, content $\gcd(A(t),c)=1$ and height $\Ht(F) = \deg
A=n$. By Theorem~\ref{Rudnick thm} the number of substitutions $a\in
\EM_h$ for which $F(a)$ is squarefree is
$$ q^h + O(nq^{h-1}) \;.
$$
Hence number of squarefrees in $I(A;h)$ is
$$\sum_{c\in \fq^\times}(q^h + O(nq^{h-1})) +O(q^h)= H + O(\frac Hq)\;,
$$
proving \eqref{asymp short int}.

\begin{comment}
We use the map $\inv_n$ which maps the short interval $I(A;h)$
bijectively onto the set of all polynomial $g$ of degree $\leq n$
(not necessarily monic) in the arithmetic progression
$g=\theta_n(A)\mod t^{n-h}$ (see Lemma~\ref{lem:bijection}).  We
will show that squarefrees are mappend bijectivly onto squarefrees
here, hence the number of squarefree polynomials in $I(A;h)$ equals
the number of squarefree polynomials in the arithmetic progression
$g=\theta_n(A)\bmod t^{n-h}$ with $\deg g\leq n$, which by
\eqref{asymp arith prog} consists of almost all of them with a
remainder of $O(H/q)$. This will prove \eqref{}.

To see that squarefrees are mapped onto squarefrees, note that if
$\deg(ab)=n$ then $\theta_n(f) = \theta_{\deg a}(a)\theta_{\deg
b}(b)$  and so if $P^2\mid f$ then $\theta_{\deg P}(P)^2 \mid
\theta_n(f)$. If $P$ is prime with $P(0)\neq 0$ (so $P\neq ct$) then
$ P^*:=\theta_{\deg P}(P)$ is also prime, of the same degree as that
of $P$. Thus the map $P\mapsto P^*$ permutes the set of all primes
(not necessarily monic) which are coprime to $t$, i.e. other than
$ct$. Moreover it is an involution when restruced to polynomials
coprime to $t$. Thus for such primes $P$, $\theta_n(f)$ is divisible
by $P*^2$ if and only if $P^2\mid f$.
\end{comment}

\section{Asymptotics for M\"obius sums}\label{sec:asymp}

 In this section we deal with cancellation in the individual sums
 $$\EN_\mu(A;h)  = \sum_{f\in I(A;h)} \mu(f) \;.
 $$
 Note that the interval $I(A;h)$ consists of all polynomials of the
 form $A+g$, where $g\in \mathcal P_{\leq h}$ is the set of all
 polynomials of degree at most $h$.

 \subsection{Small $h$}
We first point out that for  $h=0,1$ there need not be any
cancellation. We  recall Pellet's formula for the discriminant (in
odd characteristic)
\begin{equation}\label{Pellet}
\mu(f) = (-1)^{\deg f} \chi_2(\disc f)
\end{equation}
where $\chi_2:\fq^\times \to \{\pm 1\}$ is the quadratic character
of $\fq$ and $\disc f$ is the discriminant of $f$.
 From Pellet's formula we find (as in \cite{CR})
\begin{equation}
\EN_\mu(A;h) = (-1)^{\deg A} \sum_{g\in \mathcal P_{\leq h}}
\chi_2(\disc(A+g)) \;.
\end{equation}

Let $A(t)=t^n$. The discriminant of the trinomial $t^n+at+b$ is (see
e.g. \cite{Swan})
\begin{equation}
\disc(t^n+at+b) = (-1)^{n(n-1)/2} \Big(  n^nb^{n-1}+ (1-n)^{n-1} a^n
\Big) \;.
\end{equation}
Hence for the interval $I(t^n;1) = \{t^n+at+b:a,b\in \fq\}$ we
obtain
\begin{equation}
\EN_\mu(t^n;1) = (-1)^n \chi_2(-1)^{n(n-1)/2} \sum_{a,b\in \fq}
\chi_2 \Big(  n^nb^{n-1}+ (1-n)^{n-1} a^n \Big) \;.
\end{equation}
Therefore if $q=p^k$ with $p$ an odd prime  and  $2p\mid n$ then
\begin{equation}
\EN_\mu(t^n;1) =\chi_2(-1)^{n/2} q\sum_{a\in \fq} \chi_2(a)^n = \pm
q(q-1)
\end{equation}
so that $| \EN_\mu(t^n;1)|\gg q^2=H$.  A similar construction also
works for $h=0$.

\subsection{Large $h$}
We also note that for $h=n-2$, and $p\nmid n$ ($p$ is the characteristic of $\fq$) the
M\"obius sums all coincide. This is because $\mu(f(t)) =
\mu(f(t+c))$ if $\deg f\geq 1$. Therefore
$$\EN_\mu(A(t);h) = \EN_\mu(A(t+c);h)\;.$$
Now if $A(t) = t^n+a_{n-1}t^{n-1}+\dots$ is the center of the
interval, then choosing $c=-a_{n-1}/n$ gives
$$
A(t+c) = t^n+\tilde a_{n-2}t^{n-2}+\dots
$$
which contains no term of the form $t^{n-1}$. Therefore if $h=n-2$
then
$$\EN_\mu(t^n+a_{n-1}t^{n-1};n-2) = \EN_\mu(t^n;n-2)$$
has just one possible value.

Thus we may assume that $h\leq n-3$.

We note that the same is true for the squarefree case.

\subsection{Proof of Theorem~\ref{prop asymp mobius
short}} Now we prove Theorem~\ref{prop asymp mobius short}; that
is, we show that for $h\geq 2$, for any (monic) $A(t)$ of degree $n$,
\begin{equation}
\sum_{a\in \mathcal P_{\leq h}} \mu(A+a) \ll \frac{H}{\sqrt{q}}
\end{equation}
Writing $a(t) = a_ht^h+\dots +a_1t+b$, it suffices to show that
there is a constant $C=C(n,h)$ (independent of $A$ and $\vec a=(a_1,\dots,a_h)$) so that for ``most" choices of $\vec a$, (i.e. for all but
$O(q^{h-1})$) we have
\begin{equation}
\left| \sum_{b\in \fq} \mu(A(t) + a_ht^h+\dots +a_1t+b) \right|\leq
C\sqrt{q}.
\end{equation}

Using Pellet's formula, we need to show that for most $\vec a$,
\begin{equation}\label{bd for char sum}
\left| \sum_{b\in \fq}\chi_2\left( \disc(A(t) + a_ht^h+\dots
+a_1t+b)\right) \right|\leq C\sqrt{q}
\end{equation}
Now $D_a(b):= \disc(A(t) + a_ht^h+\dots +a_1t+b)$ is a polynomial in
$b$, of degree $\leq n-1$, and if we show that for most $\vec a$
 it is non-constant and squarefree
then by Weil's theorem we will get that \eqref{bd for char sum}
holds for such $\vec a$'s, with $C=n-2$. The argument in
\cite[Section 4]{CR} works verbatim here to prove that. \qed

An alternative argument is to use the work of Bank, Bary-Soroker and
Rosenzweig \cite{BBR} who prove equidistribution of cycle types of
polynomials in any short interval $I(A;h)$ for $2\leq h\leq n-2$ and
$q$ odd (this also uses \cite{CR}). Now for $f\in \EM_n$ squarefree,
$\mu(f) = (-1)^n \sign(\sigma_f)$ where $\sigma_f\subset S_n$ is the
conjugacy class of permutations induced by the Frobenius acting on
the roots of $f$, and $\sign$ is the sign character. For any $f\in
\EM_n$,  not necessarily squarefree, we denote by
$\lambda(f)=(\lambda_1,\dots,\lambda_n)$ the cycle structure of $f$,
which for squarefree $f$ coincides with the cycle structure of the
permutation $\sigma_f$. Then $\sign(\sigma_f) = \sign(\lambda(f)) :=
\prod_{j=1}^n(-1)^{(j-1)\lambda_j}$. Thus
\begin{equation}
\sum_{f\in I(A;h)} \mu(f) = (-1)^n \sum_{\substack{f\in I(A;h)\\\rm
squarefree}}\sign(\sigma_f)
\end{equation}
By Theorem~\ref{thm asymp sf},
%the results in \cite{Rudnick} (applied to the polynomials $f(x,t)
%= A(t)+cx$, $c\in \fq^\times$),
all but $O_n(H/q)$ of the
polynomials in the short interval $I(A;h)$ are squarefree, hence
\begin{equation}
\begin{split}
\sum_{\substack{f\in I(A;h)\\\rm squarefree}}\sign(\sigma_f)
 &=\sum_{f\in I(A;h)}\sign(\sigma_f) +
 O(\frac Hq)\\
& =H\left( \frac{1}{n!}\sum_{\sigma\in S_n}\sign(\sigma) +
O(\frac{1}{\sqrt{q}})\right) = O(\frac{H}{\sqrt{q}})
\end{split}
\end{equation}
by  equidistribution of cycle types in short intervals \cite{BBR},
and recalling that $\sum_{\sigma\in S_n} \sign(\sigma)=0$ for $n>1$.

%\newpage

\section{Variance in arithmetic progressions: General theory}

%Let $\fq$ be a finite field of $q$ elements and $\fq[t]$ the ring of
%polynomials with coefficients in $\fq$. Let    $\mathcal M_n$ be the
%set of monic polynomials of degree $n$. For a nonzero polynomial
%$f$, set $|f|:=q^{\deg f}$. \marginpar{check for redundancies}

%\marginpar{$T$ vs $t$ vs $x$}

%\marginpar{Define symmetric, cycle structure} \marginpar{at this
%point we don't assume that it is symmetric}

Let $\alpha:\fq[t]\to \C$ be a function on polynomials, which is
``even" in the sense that
$$\alpha(cf) = \alpha(f)$$ for the units
$c\in \fq^\times$. We assume that
\begin{equation}
\max_{\deg f\leq n} |\alpha(f)|\leq A_n
\end{equation}
with $A_n$ independent of $q$.  We will require some further
constraints on $\alpha$ later on.

We denote by $\ave{\alpha}_n$ the mean value of
$\alpha$ over all monic polynomials of degree $n$:
\begin{equation}
\ave{\alpha}_n:= \frac 1{q^n} \sum_{f\in \EM_n} \alpha(f)\;.
\end{equation}

%We will check that for cycle functions, that if $\alpha(f) =
%a(\lambda(f))$ for some map $a$ on partitions, the mean value
%$\ave{\alpha}_n$ is asymptotic, as $q\to \infty$, to the expected
%value of $a$ over all permutations of $n$ letters:
%\begin{equation}
%\lim_{q\to \infty} \ave{\alpha}_n  = \E_{S_n}(a) = \frac
%1{\#S_n}\sum_{\sigma\in S_n} a(\lambda(\sigma))
%\end{equation}

Let $Q\in \fq[t]$ be squarefree, of positive degree.
%\footnote{ We will assume $Q$ is squarefree, or even prime}.
For an arithmetic function
$\alpha:\fq[t]\to \C$,  we define its mean value over coprime residue
classes by
\begin{equation}
\ave{\alpha}_Q :=\frac 1{\Phi(Q)}\sum_{\substack{A\bmod Q\\
\gcd(A,Q)=1}} \alpha(A) \;.
\end{equation}

The sum of $\alpha$ over all monic
%\footnote{If we don't assume $\alpha$ is even then one should
%sum over all polynomials, not necessarily monic.}
polynomials of degree $n$ lying in the arithmetic progressions $f=A
\mod Q$ is
\begin{equation}
\ES_{\alpha,n,Q}(A):=\sum_{\substack{f\in \mathcal M_n\\f=A\bmod Q}}
\alpha(f) \;.
\end{equation}
We wish to study the fluctuations in $\ES(A)$ as we vary $A$ over
residue classes coprime to $Q$.
%\marginpar{do we consider and only (A,Q)=1}
The mean value of $\ES$ is
\begin{equation}
\ave{\ES_\alpha}_Q =\frac 1{\Phi(Q)} \sum_{\substack{f\in \mathcal M_n\\
 (f,Q)=1}} \alpha(f)
 % =\frac{q^n}{\Phi(Q)} \left( \ave{\alpha}_n +O( \frac {\max_{  \mathcal M_n} |\alpha|}{\Phi(Q)} ) \right)    ??
 \end{equation}
%\marginpar{for $Q$ prime}
where $\Phi(Q)$ is the number of invertible residues modulo $Q$.

Our goal is to compute the variance
\begin{equation}
\Var_Q (\ES_\alpha) = \frac 1{\Phi(Q)}\sum_{\substack{A\bmod Q\\
(A,Q)=1}} |\ES_\alpha(A) - \ave{\ES_\alpha}|^2 \;.
\end{equation}

%Our goal is to evaluate the variance in a number of cases, when
%$\alpha$ is the divisor function $\divid$, the M\"obius function
%$\mu$, its square $\mu^2$ (counting squarefrees in arithmetic
%progressions). Possibly even the generalized von-Mangoldt function
%$\Lambda_r$, counting almost-primes.

%Using the orthogonality relation for characters modulo $Q$ gives, if $(A,Q)=1$,
%\begin{equation}\label{apply orthogonality gen}
%\ES_{\alpha,n,Q}(A) = \frac 1{\Phi(Q)} \sum_{\substack{f\in \mathcal M_n\\
 %(f,Q)=1}} \alpha(f)  + \frac{1}{\Phi(Q)} \sum_{\chi\neq
%\chi_0}\overline{\chi(A)}\EM(n;\alpha\chi)
%\end{equation}
%where
%\begin{equation}
%\EM(n;\alpha\chi) :=\sum_{f\in \mathcal M_n}\chi(f)\alpha(f)
%\end{equation}
%Thus the mean value over all $A\bmod Q$ (coprime to $Q$) is
%\begin{equation}
%\ave{\ES} =\frac 1{\Phi(Q)} \sum_{\substack{f\in \mathcal M_n\\
 %(f,Q)=1}} \alpha(f)
 %\end{equation}

\subsection{A formula for the variance}
%Let $Q\in \fq[t]$ be a polynomial of positive degree, which we
%assume henceforth is squarefree.

%We study the sums over arithmetic progressions  $f=A\bmod Q$, with
%$A$ coprime to $Q$:
%\begin{equation}
%\ES(A) = \ES_{\alpha,n,Q}(A) = \sum_{\substack{f\in \EM_n \\
%f=A\bmod Q}} \alpha(f)
%\end{equation}

Expanding in Dirichlet characters modulo $Q$ gives
\begin{equation}\label{expand S in dir}
\ES(A) = \frac 1{\Phi(Q)} \sum_{\chi \bmod Q} \bar\chi(A)
\EM(n;\alpha\chi)
\end{equation}
where
\begin{equation}
\EM(n;\alpha\chi) :=\sum_{f\in \mathcal M_n}\chi(f)\alpha(f) \;.
\end{equation}

The mean value is the contribution of the trivial character
$\chi_0$:
\begin{equation}
\ave{\ES}_Q =\frac 1{\Phi(Q)}\sum_{\substack{ f\in \EM_n\\
\gcd(f,Q)=1}} \alpha(f)
\end{equation}
so that
\begin{equation}\label{apply orthogonality gen}
\ES(A)- \ave{\ES}_Q = \frac 1{\Phi(Q)} \sum_{\chi \neq \chi_0\bmod
Q} \bar\chi(A) \EM(n;\alpha\chi)\;.
\end{equation}

Inserting \eqref{expand S in dir} and using the orthogonality relations for Dirichlet characters
as in \cite{KR},  we see that the
variance is
\begin{equation}\label{gen formula for varianceAP}
\Var_Q(\ES) =\ave{\left|\ES-\ave{\ES}_Q \right|^2}_Q= \frac
1{\Phi(Q)^2} \sum_{\chi\neq \chi_0} |\EM(n;\alpha\chi)|^2 \;.
\end{equation}

\subsection{Small $n$}\label{sec: small $n$}%{sec:small n}

% We will deal with $n\geq \deg Q$, since otherwise each arithmetic progression modulo $Q$ contains at most one element of degree $n$
%and in this case (see \S~\ref{sec: small $n$})
%\begin{equation}\label{elt claim}
%\Var_Q(\ES) \sim  \frac{q^n}{\Phi(Q)} \ave{\alpha^2}_n  ,\quad
%n<\deg Q
%\end{equation}

 If $n<\deg Q$, then
%the problem is both boring and trivial, since
there is at most \underline{one} $f$ with $\deg f=n$ and $f=A\bmod
Q$,
and in this case
%(see \S~\ref{sec: small $n$})
\begin{equation}\label{elt claim}
\Var_Q(\ES) \sim  \frac{q^n}{\Phi(Q)} \ave{\alpha^2}_n  ,\quad
n<\deg Q \;.
\end{equation}

Indeed, if $n<\deg Q$, then
\begin{equation}
|\ave{\ES}_Q|   =| \frac 1{\Phi(Q)}\sum_{\substack{f\in \EM_n\\
\gcd(f,Q)=1}} \alpha(f)|\leq  \frac 1{\Phi(Q)}\sum_{f\in \EM_n}
|\alpha(f)| \leq \frac{A_nq^n}{\Phi(Q)} \ll_n \frac 1q\;.
\end{equation}
%Moreover if $Q$ is prime and $\deg
%f=n<\deg Q$ then automatically $f$ is co-prime to $Q$. Further, if
%we assume that $n\leq (1-2\delta)\deg Q$ then the mean value is
%\begin{equation}
%\ave{\ES} =q^{-\delta n}
%\end{equation}
Hence
\begin{equation*}
\begin{split}
\Var_Q(\ES) &= \frac 1{\Phi(Q)} \sum_{\substack{A\bmod Q\\\gcd(A,Q)=1}} |\ES(A)|^2 (1+O(q^{-1})) \\
& =\frac 1{\Phi(Q)}\sum_{\substack{f\in \mathcal M_n\\ \gcd(f,Q)=1}}|\alpha(f)|^2 (1+ O(q^{-1})) \\
& = \frac {q^n}{\Phi(Q)}\ave{|\alpha|^2}_n(1+O(\frac 1q))
\end{split}
\end{equation*}
as claimed.

%and so in this case
%\begin{equation}\label{elt claim}
%\Var_Q(\ES) = \frac{q^n}{\Phi(Q)} \ave{|\alpha|^2}_n (1+O(q^{-1})),
%\quad n<\deg Q
%\end{equation}

%\section{A general formula for the variance in short intervals}

%************
%\newpage

\section{Variance in short intervals: General theory}

%\subsection{Short intervals}

 %For $A\in \mathcal M_n$ of
%degree $n$, and $h<n$, we define a ``short interval around $A$ of
%length $q^h$'' by
%\end{equation}
%where the norm of a polynomial $0\neq f\in \fq[t]$ is
%$||f||:=q^{\deg f}$.
%Note that
%$$ \#I(A;h) = q^{h+1}=:H \;. $$

Given an arithmetic function $\alpha:\fq[t]\to \C$,
%be a function on polynomials, which is even, symmetric and weakly multiplicative.
define its sum on short intervals as
\begin{equation}
\EN_\alpha(A;h)  = \sum_{ f\in I(A;h) }\alpha(f) \;.
\end{equation}
The mean value of $\EN_\alpha$ is (see Lemma~\ref{lem:Mean value})
\begin{equation}
\ave{\EN_\alpha(\bullet, h)} = q^{h+1}\ave{\alpha}_n =H\ave{\alpha}_n \;. 
\end{equation}
Our goal will be to compute the variance of $\EN_\alpha$,
\begin{equation}
\Var{\EN_\alpha} = \frac 1{q^n}\sum_{A\in \mathcal M_n}
|\EN_\alpha(A) - \ave{\EN_\alpha}|^2
\end{equation}
and more generally, given two such functions $\alpha,\beta$, to
compute the covariance
 \begin{equation}
\Cov(\EN_\alpha,\EN_\beta)
=\ave{\Big(\EN_\alpha-\ave{\EN_\alpha}\Big)\Big(\EN_\beta-\ave{\EN_\beta}\Big)} \;.
\end{equation}

\subsection{Background on short intervals (see  \cite{KR})}
\label{sec:backgd on si}
 %Let $\fq$ be a finite field of $q$ elements and $\fq[t]$ the ring of polynomials with
%coefficients in $\fq$.
Let
\begin{equation}
\mathcal P_n=\{f\in \fq[t]:\deg f=n\}
\end{equation}
be the set of polynomials of degree $n$,
\begin{equation}
  \mathcal P_{\leq n} = \{0\} \cup \bigcup _{0\leq m\leq n}\mathcal P_m
\end{equation}
the space of polynomials of degree at most $n$ (including $0$),
 and $\mathcal M_n\subset \mathcal P_n$ the subset of monic polynomials.

%The norm of a polynomial $0\neq f\in \fq[t]$ is
%\begin{equation}
%||f||:=q^{\deg f}
%\end{equation}
%For $A\in \mathcal P_n$ of degree $n$, and $h<n$, we define a
%``short interval around $A$ of length $q^h$'' by
%\begin{equation}
%I(A;h):=\{f: ||f-A||\leq q^h \} = \{f: \deg(f-A)\leq h \}
%\end{equation}
%so that
By definition of short intervals,
\begin{equation}
I(A;h)=A+  \mathcal P_{\leq h}\;,
\end{equation}
and hence
\begin{equation}
\#I(A;h) = q^{h+1}=:H\;.
\end{equation}

For $h=n-1$, $I(A;n-1) = \EM_n$ is the set of all monic polynomials
of degree $n$.  For $h\leq n-2$, if $||f-A||\leq q^h$ then $\deg
f=\deg A $ and $A$ is monic if and only if $f$ is monic. Hence for
$A$ monic, $I(A;h)$ consists of only monic polynomials and all monic
$f$'s of degree $n$ are contained in one of the intervals $I(A;h)$
with $A$ monic of degree $n$. Moreover,
 \begin{equation}\label{coincidence}
   I(A_1;h)\bigcap I(A_2;h)\neq \emptyset \leftrightarrow
   \deg(A_1-A_2)\leq h \leftrightarrow I(A_1;h) = I(A_2;h)
 \end{equation}
and we get a partition of $\mathcal P_n$ into disjoint ``intervals''
parameterized by $B\in \mathcal P_{n-(h+1)}$:
\begin{equation}
  \mathcal P_n = \coprod_{B\in  \mathcal P_{n-(h+1)}} I(t^{h+1}B;h)
\end{equation}
and likewise for monics (recall $h\leq n-2$):
\begin{equation}
  \mathcal M_n = \coprod_{B\in  \mathcal M_{n-(h+1)}} I(t^{h+1}B;h)
\end{equation}

\subsection{An involution}
Let $n\geq 0$. We define a map $\inv_n: \mathcal P_{\leq n}\to
\mathcal P_{\leq n}$ by
$$ \inv_n(f)(t) = t^n f(\frac 1t)$$
which takes $f(t) = f_0+f_1t+\dots +f_n t^n$, $n=\deg f$ to the
``reversed'' polynomial
\begin{equation}
  \inv_n(f)(t) = f_0t^n+f_1t^{n-1}+\dots +f_n\;.
\end{equation}
 For $0\neq f\in \fq[t]$ we define
\begin{equation}
  f^*(t):=t^{\deg f} f(\frac 1t)
\end{equation}
so that $\inv_n(f) = f^*$ if $f(0)\neq 0$. Note that if $f(0)=0$
then this is false, for example $(t^k)^*=1$ but $\inv_n(t^k) =
t^{n-k}$ if $k\leq n$.

We have  $\deg \inv_n(f)\leq n$ with equality if and only if
$f(0)\neq 0$.  Moreover for $f\neq 0$, $f^*(0)\neq 0$ and $f(0)\neq
0$ if and only if $\deg f^* = \deg f$. Restricted to polynomials
which do not vanish at $0$, equivalently are co-prime to $t$, then
$*$ is an involution:
\begin{equation}
f^{**} = f, \quad  f(0)\neq 0\;.
\end{equation}
We  also have multiplicativity:
\begin{equation}
  (fg)^* = f^* g^*\;.
\end{equation}

 The map $\inv_m$ gives a bijection
\begin{equation}\label{eq:biject1}
  \begin{split}
  \inv_m: \mathcal M_{m} & \to
  \{  C \in \mathcal P_{\leq m}: C(0) =1 \}\\
     B &\mapsto  \inv_m(B)
  \end{split}
\end{equation}
with polynomials of degree $\leq m$ with constant term $1$. Thus as
$B$ ranges over $\mathcal M_{m}$, $\inv_m(B)$ ranges over all
invertible residue class $C \mod t^{m+1}$ so that $C(0)=1$.

\subsection{Short intervals as arithmetic progressions modulo $t^{n-h}$}
Suppose $h\leq n-2$. Define the arithmetic progression
\begin{equation}
\mathcal P_{\leq n}(t^{n-h};C) =\{ g\in \mathcal P_{\leq n}: g\equiv
C \mod t^{n-h} \} = C+t^{n-h} \mathcal P_{\leq h} \;.
\end{equation}
Note that the progression contains $q^{h+1}$ elements.
\begin{lem}\label{lem:bijection}
Let $h\leq n-2$ and $B\in \mathcal M_{n-h-1}$. Then the  map
$\inv_n$ takes the ``interval" $I(t^{h+1}B;h)$ bijectively onto the
arithmetic progression $\mathcal P_{\leq
n}(t^{n-h};\inv_{n-h-1}(B))$, with those $f\in I(t^{h+1}B;h)$ such
that $f(0)\neq 0$ mapping onto those $g\in \mathcal P_{\leq
n}(t^{n-h};\inv_{n-h-1}(B))$ of degree exactly $n$.
\end{lem}
\begin{proof}
We first check that $\inv_n$ maps the interval $I(t^{h+1}B;h)$  to
the arithmetic progression $\mathcal P_{\leq
n}(t^{n-h};\inv_{n-h-1}(B))$. Indeed if  $B=b_0+\dots
+b_{n-h-1}t^{n-h-1}$, with $b_{n-h-1}=1$, and $f=f_0+\dots
+f_nt^n\in I(t^{h+1}B;h)$ then
\begin{equation}
f=f_0 +\dots + f_ht^h + t^{h+1}(b_0+\dots +b_{n-h-1}t^{n-h-1})
\end{equation}
so that
\begin{equation}
\begin{split}
\inv_n(f) &=f_0t^n+\dots+f_ht^{n-h} + b_0t^{n-h-1}+\dots +
b_{n-h-1}\\
&= \inv_{n-h-1}(B)\mod t^{n-h} \;.
\end{split}
\end{equation}
Hence $\inv_n(f)\in\mathcal P_{\leq n}(t^{n-h};\inv_{n-h-1}(B))$.

Now the map $\inv_n: \mathcal P_{\leq n}\to \mathcal P_{\leq n}$ is
a bijection, and both $I(t^{h+1}B;h)$ and $\mathcal P_{\leq
n}(t^{n-h};\inv_{n-h-1}(B))$   have size $q^{h+1}$, and therefore
$\inv_n:I(t^{h+1}B;h) \to \mathcal P_{\leq n}(t^{n-h};B^*) $ is a
bijection.
\end{proof}

\subsection{The mean value}

%Let $\alpha:\fq[t]\to \C$ be a function on polynomials. For $h\leq
%n-2$, $A\in \mathcal M_n$, define the sum of $\alpha$ on the short
%interval $I(A;h)$   by
%\begin{equation}
%\EN_\alpha(A;h)  = \sum_{f\in I(A;h)}\alpha(f)\;.
%\end{equation}

%\begin{verbatim} no condition on f(0)  !!!!\end{verbatim}

\begin{lem} \label{lem:Mean value}
The mean value of
$\EN_\alpha(\bullet;h)$ over $\EM_n$ is
\begin{equation}
 \ave{\EN_\alpha (\bullet;h)}=q^{h+1} \frac 1{q^n} \sum_{f\in \EM_n} \alpha(f) = H \ave{\alpha}_n \;.
\end{equation}
\end{lem}
%Note that if $\alpha$ is even, then we may replace the mean over
%monic polynomials $\EM_n$ of degree $n$ by the mean over all
%polynomials of degree $n$. \marginpar{define "even"?}
\begin{proof}
  From the definition, we have
\begin{equation}
\begin{split}
  \ave{\EN_\alpha (\bullet;h)} &= \frac 1{\#\mathcal M_{n-h-1}} \sum_{B\in
    \mathcal M_{n-h-1}}
\EN_\alpha(t^{h+1}B;h) \\
&=\frac 1{q^{n-h-1}}\sum_{B\in
    \mathcal M_{n-h-1}} \sum_{f\in I(t^{h+1}B;h)} \alpha(f) \\
&=q^{h+1} \frac 1{q^n} \sum_{f\in \EM_n} \alpha(f) = q^{h+1}
\ave{\alpha}_n \;.
\end{split}
\end{equation}
\end{proof}

\subsection{A class of arithmetic functions}
%Let $\fq$ be a finite field of $q$ elements and $\fq[t]$ the ring of
%polynomials with coefficients in $\fq$. Let    $\mathcal M_n$ be the
%set of monic polynomials of degree $n$. For a nonzero polynomial
%$f$, set $|f|:=q^{\deg f}$.

Let $\alpha:\fq[t]\to \C$ be a function on polynomials, which is:

\begin{itemize}
\item
 {\bf Even}  in the sense that
\begin{equation*}
 \alpha(cf) = \alpha(f),\quad c\in \fq^\times \;.
\end{equation*}

\item {\bf Multiplicative}, that is $\alpha(fg) = \alpha(f)\alpha(g)$ if $f$ and $g$ are co-prime.
In fact we will only need a weaker condition , ``weak  multiplicativity":
  If  $f(0)\neq 0$,  i.e.  $\gcd(f,t)=1$ then
\begin{equation*}
\alpha(t^kf) = \alpha(t^k) \alpha(f), \quad f(0)\neq 0 \;.
\end{equation*}

\item
{\bf Bounded}, that is it satisfies the growth condition
\begin{equation*}
\max_{f\in \EM_n} |\alpha(f)|\leq A_n
\end{equation*}
independently of $q$.

\item
 {\bf  Symmetric} under the  map  $f^*(t):=t^{\deg f}f(\frac 1t)$,
\begin{equation*}\label{inv under*}
\alpha(f^*) = \alpha(f), \quad f(0)\neq 0 \;.
\end{equation*}

\end{itemize}

Examples are the M\"obius function $\mu$, its square $\mu^2$ which
is the indicator function of squarefree-integers,
%the indicator function of the $r$-free integers,
and the divisor functions (see \cite{KRR}).

Note that multiplicativity (and the ``weak multiplicativity" condition) excludes the case of
the von Mangoldt function, treated in \cite{KR}  where we
are counting prime polynomials in short intervals or arithmetic
progressions. A related case of almost primes was treated by Rodgers
\cite{Rodgers}.

\subsection{A formula for $\EN_\alpha(A;h)$ }

%Consequently from Lemmas~\ref{Gamma in terms of Psi} and \ref{lem:Mean value} we find that  the centered variable is

We present a useful formula for the short interval sums $\EN_\alpha(\bullet,h)$ in terms of sums over even Dirichlet characters modulo $t^{n-h}$.
Recall that a Dirichlet character $\chi$  is ``even" if $\chi(cf) = \chi(f)$ for all scalars $c\in \fq^\times$, and we say that $\chi$ is ``odd" otherwise.   The number of even characters modulo
$t^m$ is $\Phi_{ev}(t^m)=q^{m-1}$.
We denote by $\chi_0$ the trivial character.

\begin{lem}\label{Gamma in terms of Psi}
If $\alpha:\fq[t]\to \C$ is even, symmetric and weakly
multiplicative, and $0\leq h\leq n-2$,  then for all $B\in
\EM_{n-h-1}$,
\begin{multline}\label{centered gamma}
\EN_\alpha(t^{h+1}B;h)=\ave{\EN_\alpha (\bullet;h)} \\
+ \frac{1}{\Phi_{ev}(t^{n-h})} \sum_{m=0}^n \alpha(t^{n-m})
\sum_{\substack{\chi \bmod t^{n-h}\\\chi\neq \chi_0\mbox{ even}}}
\bar\chi(\inv_{n-h-1}(B)) \EM(m;\alpha\chi)
\end{multline}
%
%\begin{multline}\label{eq:Gamma in terms of Psi}
%\EN_\alpha(t^{h+1}B;h) = q^{h+1} \frac 1{q^n }\sum_{f\in \mathcal
%M_n} \alpha(f) \\ + \frac {1}{\Phi_{ev}(t^{n-h})}
 % \sum_{\substack{\chi \bmod t^{n-h}\\\chi \neq \chi_0 \,{\rm even}}} \bar\chi(\inv_{n-h-1}(B))
%\sum_{m=0}^n \alpha(t^{n-m}) \EM(m;\alpha\chi)
%\end{multline}
where
\begin{equation}
\EM(n;\alpha\chi)=\sum_{f\in \mathcal M_n}\alpha(f)\chi(f) \;.
\end{equation}
\end{lem}
%We note that $\Phi_{ev}(t^m)=q^{m-1}$.
\begin{proof}
 Writing each $f\in \mathcal M_n$ uniquely as $f=t^{n-m}f_1$ with $f_1\in \mathcal M_m$ and $f_1(0)\neq
 0$, for which $\inv_n(f) = \inv_m(f_1) = f_1^*$,
 we obtain, using (weak) multiplicativity, 
\begin{equation}
\begin{split}
\EN_\alpha(t^{h+1}B;h) & = \sum_{m=0}^n \sum_{\substack{f_1\in
\mathcal M_m
\\  f_1(0) \neq 0 \\ t^{n-m}f_1 \in I(t^{h+1}B;h) }}
\alpha(t^{n-m} f_1 )  \\
&=\sum_{m=0}^n   \alpha(t^{n-m}) \sum_{\substack{f_1\in \mathcal M_m \\
f_1(0) \neq 0 \\ t^{n-m}f_1 \in I(t^{h+1}B;h) }} \alpha(f_1 ) \;.
\end{split}
\end{equation}
Since $f_1(0)\neq 0$, we have  that $f_1^*=\inv_m(f_1)$ runs over all
polynomials $g$ of degree $m$ (not necessarily monic) so that
$g\equiv  \inv_{n-h-1}(B) \mod t^{n-h}$ by
Lemma~\ref{lem:bijection}, and moreover $\alpha(f_1) =
\alpha(f_1^*)=\alpha(\inv_m(f_1))$. Hence
\begin{equation}
%\begin{split}
\EN_\alpha(t^{h+1}B;h)  = \sum_{m=0}^n   \alpha(t^{n-m})
\sum_{\substack{\deg g=m\\ g\equiv \inv_{n-h-1}(B) \mod t^{n-h}}}
\alpha(g) \;.
%\end{split}
\end{equation}

 Using characters to pick out the conditions
$g\equiv  \inv_{n-h-1}(B)\bmod t^{n-h}$ (note that since $B$ is
monic,  $\inv_{n-h-1}(B)$ is coprime to $t^{n-h}$) gives
\begin{equation}
  \sum_{\substack{ \deg g=m \\ g\equiv \inv_{n-h-1}(B)\mod t^{n-h}}}
\alpha(g)= \frac 1{\Phi(t^{n-h})} \sum_{\chi \bmod t^{n-h}}
\bar\chi(\inv_{n-h-1}(B)) \tilde \EM(m;\alpha\chi)
\end{equation}
where
\begin{equation}
\tilde \EM(m;\alpha\chi)=\sum_{ \deg g=m } \chi(g)\alpha(g) \;,
\end{equation}
the sum running over all $g$ of degree $m$.

Since $\alpha$ is even, we find that
\begin{equation*}
\begin{split}
\tilde \EM(m;\alpha\chi)&=\sum_{f\in \mathcal M_m}\sum_{c \in
\fq^\times} \chi(cf)\alpha(cf)\\
& =\sum_{f\in \mathcal M_m}\alpha(f)\chi(f)\sum_{c \in \fq^\times}
\chi(c) \\
&= \begin{cases} (q-1)\sum_{f\in \mathcal M_m}\alpha(f)\chi(f),&
\chi \; {\rm even}\\0,&\chi\;{\rm odd}\end{cases}
\end{split}
\end{equation*}
where now the sum is over monic polynomials of degree $m$.

Thus we get, on noting that $\Phi(t^{n-h})/ (q-1)=
\Phi_{ev}(t^{n-h})$, that
\begin{equation}
  \sum_{\substack{ \deg g=m\\ g\equiv \inv_{n-h-1}(B) \mod t^{n-h}}}
\alpha(g)  = \frac {1}{\Phi_{ev}(t^{n-h})} \sum_{\substack{\chi
\bmod t^{n-h}\\\chi\mbox{ even}}} \bar\chi(\inv_{n-h-1}(B))
\EM(m;\alpha\chi) \;.
\end{equation}
 Therefore
\begin{equation}\label{intermediate var}
%\begin{split}
\EN_\alpha(t^{h+1}B;h)  = \sum_{m=0}^n   \alpha(t^{n-m})  \frac
{1}{\Phi_{ev}(t^{n-h})} \sum_{\substack{\chi \bmod
t^{n-h}\\\chi\mbox{ even}}} \bar\chi(\inv_{n-h-1}(B))
\EM(m;\alpha\chi) \;.
%\end{split}
\end{equation}

The trivial character $\chi_0$  contributes a term
\begin{equation}
\frac 1{\Phi_{ev}(t^{n-h})} \sum_{m=0}^n \sum_{\substack{g\in \mathcal M_m\\
g(0)\neq 0}} \alpha(t^{n-m})\alpha(g)= \frac
{q^{h+1}}{q^{n}}\sum_{f\in \mathcal M_n} \alpha(f)
\end{equation}
on using weak multiplicativity. Inserting this into
\eqref{intermediate var} and using Lemma~\ref{lem:Mean value} we
obtain the formula claimed.
\end{proof}

\subsection{Formulae for variance and covariance}
\label{sec:formulae for var}

Given an arithmetic function $\alpha$,  %satisfying \eqref{},
the variance of $\EN_\alpha$  is
 \begin{equation}
 \begin{split}
\Var(\EN_\alpha )
&=\ave{ |\EN_\alpha-\ave{\EN_\alpha}|^2} \\
&=\frac 1{q^{n-h-1}} \sum_{B\in \mathcal M_{n-h-1}}
|\EN_\alpha(t^{h+1}B;h)-\ave{\EN_\alpha }|^2
 \end{split}
\end{equation}
and likewise given two such functions $\alpha,\beta$,
%satisfying \eqref{},
the covariance of $\EN_\alpha$ and $\EN_\beta$ is
\begin{equation}
\Cov(\EN_\alpha,\EN_\beta)
=\ave{\Big(\EN_\alpha-\ave{\EN_\alpha}\Big)\Big(\EN_\beta-\ave{\EN_\beta}\Big)} \;.
\end{equation}

We use the following lemma, an extension of the argument of
\cite{KR}:
\begin{lem}\label{covariance lemma}
If $\alpha,\beta$ are even, symmetric and weakly multiplicative, and
$0\leq h\leq n-2$,  then
\begin{multline}
\Cov(\EN_\alpha,\EN_\beta) = \\   \frac {1}{\Phi_{ev}(t^{n-h})^2}
\sum_{\substack{\chi \bmod t^{n-h}\\\chi\neq \chi_0\, {\rm
even}}}\sum_{m_1,m_2=0}^n \alpha(t^{n-m_1})
\overline{\beta(t^{n-m_2}) }\EM(m_1;\alpha\chi)
 \overline{\EM(m_2;\beta\chi)}\;.
\end{multline}
\end{lem}
%In many cases, one can in turn express $\EM(n;\alpha\chi) $ in terms
%of the Dirichlet L-function $L(u,\chi)$. Then one wants to express
%the character average in terms of a matrix integral, using Katz's
%equidistribution theorem \cite{KatzKR2}.

\begin{proof}
By Lemma~\ref{Gamma in terms of Psi}, $\Cov(\EN_\alpha,\EN_\beta)$
equals
\begin{multline*}
  \frac {1}{\Phi_{ev}(t^{n-h})^2}
\sum_{\substack{\chi_1,\chi_2 \bmod t^{n-h}\\\chi_1,\chi_2\neq
\chi_0\, {\rm even}}}\sum_{m_1,m_2=0}^n \alpha(t^{n-m_1})
\overline{\beta(t^{n-m_2}) }\EM(m_1;\alpha\chi_1)
 \overline{\EM(m_2;\beta\chi_2)}
\\ \times \frac 1{q^{n-h-1}}
\sum_{B\in \mathcal M_{n-h-1} }
\bar\chi_1(\inv_{n-h-1}(B))\chi_2(\inv_{n-h-1}(B)) \;.
\end{multline*}

As $B$ runs over the monic polynomials $\mathcal M_{n-h-1}$, the
image $\inv_{n-h-1}(B)$ runs over all polynomials $C \mod t^{n-h}$
with $C(0)=1$ (see \eqref{eq:biject1}). Thus
\begin{equation}
  \sum_{B\in \mathcal M_{n-h-1} }
\bar\chi_1(\inv_{n-h-1}(B))\chi_2(\inv_{n-h-1}(B)) =   \sum_{\substack{C\mod t^{n-h}\\
C(0)=1}}\bar\chi_1(C)\chi_2(C) \;.
\end{equation}

Since $\chi_1,\chi_2$ are both even, we may ignore the condition
$C(0)=1$ and use the orthogonality relation (recall
$\Phi_{ev}(t^{n-h}) = q^{n-h-1}$)  to get (see \cite[Lemma 3.2]{KR})
\begin{equation}
\frac 1{q^{n-h-1}} \sum_{\substack{C\mod t^{n-h}\\
C(0)=1}}\bar\chi_1(C)\chi_2(C) = \delta(\chi_1,\chi_2)
\end{equation}
so that
\begin{multline}
%\begin{split}
\Cov(\EN_\alpha,\EN_\beta) =\\    \frac {1}{\Phi_{ev}(t^{n-h})^2}
\sum_{\substack{\chi \bmod t^{n-h}\\\chi\neq \chi_0\, {\rm
even}}}\sum_{m_1,m_2=0}^n \alpha(t^{n-m_1})
\overline{\beta(t^{n-m_2}) }\EM(m_1;\alpha\chi)
 \overline{\EM(m_2;\beta\chi)}
%\end{split}
\end{multline}
as claimed.
\end{proof}

%\newpage

\section{Characters, L-functions and equidistribution}

Before applying the variance formulae of \S~\ref{sec:formulae for
var}, we  survey some background on Dirichlet characters, their
L-functions and recent equidistribution theorems due to N.~Katz.

\subsection{Background on Dirichlet characters and L-functions}

Recall that a Dirichlet character $\chi$  is ``even" if $\chi(cf) = \chi(f)$ for all scalars $c\in \fq^\times$,
and we say that $\chi$ is ``odd" otherwise.   The number of even characters modulo
$t^m$ is $\Phi_{ev}(t^m)=q^{m-1}$.
We denote by $\chi_0$ the trivial character.

%\subsection{Primitive characters}
A character $\chi$ is {\em primitive} if there is no proper divisor $Q'\mid
Q$ so that $\chi(F)=1$ whenever $F$ is coprime to $Q$ and $F=1\mod
Q'$. We denote by $\Phi_{prim}(Q)$ the number of primitive characters
modulo $Q$.
% we clearly have $\Phi(Q) =\sum_{D\mid Q} \Phi_{prim}(D)$
%and hence by M\"obius inversion,
%\begin{equation}
% \Phi_{prim}(Q) = \sum_{D\mid Q} \mu(D) \Phi(\frac QD)
%\end{equation}
%the sum over all monic polynomials dividing $Q$. Therefore
%\begin{equation}
%\left| \frac{\Phi_{prim}(Q)}{\Phi(Q)}-1 \right| \leq \frac{2^{\deg
%Q}}{q} \;.
%\end{equation}
As $q\to \infty$, almost all characters are primitive in the
sense that
\begin{equation}
\frac{\Phi_{prim}(Q)}{\Phi(Q)} = 1 +O(\frac 1q)\;,
\end{equation}
the implied constant depending only on $\deg Q$.

Moreover, as $q\to \infty$ with $\deg Q$ fixed, almost
all characters are primitive and odd:
\begin{equation}
\frac{\Phi_{prim}^{odd}(Q)}{\Phi(Q)} = 1 +O(\frac 1q)\;,
\end{equation}
the implied constant depending only on $\deg Q$.

One also has available similar information about the number $\Phi_{prim}^{ev}(Q)$ of
even primitive characters. What we will need to note is that for $Q(t) = t^m$, $m\geq 2$,
\begin{equation}
\Phi_{prim}^{ev}(t^m) = q^{m-2}(q-1)\;.
\end{equation}

%\subsection{L-functions}

The L-function $L(u,\chi)$ attached to $\chi$ is defined as
\begin{equation}\label{Def of L}
 L(u,\chi) = \prod_{P\nmid Q} (1-\chi(P)u^{\deg P})^{-1}
\end{equation}
where the product is over all monic irreducible polynomials in
$\fq[t]$. The product is absolutely convergent for $|u|<1/q$. If
$\chi=\chi_0$ is the trivial character modulo $Q$, then
\begin{equation}
L(u,\chi_0) = Z(u) \prod_{P\mid Q} (1-u^{\deg P})\;,
\end{equation}
where
$$
Z(u) = \prod_{P\;{\rm prime}} (1-u^{\deg P})^{-1} = \frac 1{1-qu}
$$
 is the zeta function of $\fq[t]$. Also set $\zeta_q(s):=Z(q^{-s})$.

If $Q\in \fq[t]$ is a polynomial of degree $\deg Q\geq 2$, and
$\chi\neq \chi_0$  a nontrivial character mod $Q$, then the
L-function $L(u,\chi)$ is  a polynomial in $u$ of degree %at most
$\deg Q-1$.
 Moreover, if $\chi$ is an ``even'' character,
then there is a "trivial" zero at $u=1$.

We may factor $ L(u,\chi)$  in terms of the inverse roots
\begin{equation}
L(u,\chi) =\prod_{j=1}^{\deg Q-1}(1-\alpha_j(\chi)u) \;.
\end{equation}
%In particular, the coefficient $A_1(\chi)$ of $u$ is
%$$\sum_{\deg P=1}\chi(P) = -\sum_{j=1}^{\deg Q-1} \alpha_j(\chi)\;.$$
%where the sum is over monic polynomials of degree one (necessarily
%prime).
The Riemann Hypothesis, proved by Andre Weil (1948),
%(some cases were proved  previously by Hasse in the 1930's),
is that for each (nonzero) inverse root, either $\alpha_j(\chi)=1$
or
\begin{equation}\label{eqRHWeil}
|\alpha_j(\chi)| = q^{1/2} \;.
\end{equation}

If $\chi$ is a {\em primitive} and {\em odd} character modulo $Q$, then all inverse roots $\alpha_j$ have absolute value $\sqrt{q}$, and for $\chi$ primitive and {\em even} the same holds except for the trivial zero at $1$. We then
  write the nontrivial inverse roots as $\alpha_j=q^{1/2}e^{i\theta_j}$ and define a unitary matrix
\begin{equation}
 \Theta_\chi =  \mbox{diag}(e^{i\theta_1},\dots,e^{i\theta_N} )\;.
\end{equation}
which determines a unique conjugacy class in the unitary group $U(N)$, where
  $N=\deg Q-1$ for $\chi$ odd, and  $N=\deg Q-2$ for $\chi$ even.
The unitary matrix $\Theta_\chi$ (or rather, the conjugacy class of
unitary matrices) is called the unitarized Frobenius matrix of
$\chi$.

\subsection{Katz's equidistribution theorems}

Crucial ingredients in our results on the variance are equidistribution and independence results
for the Frobenii $\Theta_\chi$ due to N.~Katz.

\begin{thm} \label{thm:KatzKR even}
i) \cite{KatzKR2}
Fix\footnote{If the characteristic of $\fq$ is different than $2$ or
  $5$ then the result also holds for $m=3$.} $m\geq 4$. The unitarized
Frobenii $\Theta_\chi$ for the family of even primitive characters
mod $T^{m+1}$ become equidistributed in the projective unitary group
$PU(m-1)$ of size $m-1$, as $q\to \infty$.

ii) \cite{KatzKR3} If  $m\geq 5$ and in addition the $q$'s are
coprime to $6$, then
 the set of pairs of conjugacy classes $(\Theta_\chi, \Theta_{\chi^2})$
become equidistributed in the space of conjugacy classes of the product $PU(m-1)\times PU(m-1)$.
\end{thm}

For odd characters, the corresponding equidistribution and independence results are

\begin{thm}\label{thm:KatzKR odd}

i) \cite{KatzKR1}
Fix $m\geq 2$. Suppose we are given a sequence of finite fields
$\fq$ and squarefree polynomials $Q(T)\in \fq[T]$ of degree $m$. As
$q\to \infty$, the conjugacy classes $\Theta_{\chi}$ with $\chi$
running over all primitive odd characters modulo $Q$, are uniformly
distributed in the unitary group $U(m-1)$.

ii) \cite{KatzKR4} If in addition we restrict to $q$ odd, then the
set of pairs of conjugacy classes $(\Theta_\chi, \Theta_{\chi^2})$
become equidistributed in the space of conjugacy classes of the
product $U(m-1)\times U(m-1)$.
\end{thm}

%\newpage

\section{Variance of the  M\"obius function in short intervals}
\label{sec:var mob si}

%\subsection{Sums of the M\"obius function in short intervals}

%The M\"obius function $\mu(F)$ is clearly even, and is {\em symmetric}.
%Set
%\begin{equation}
%\EN_\mu(A;h):= \sum_ {f\in I(A;h)} \mu(f)
%\end{equation}
 % \marginpar{change all below to allow f(0)=0}
%According to Lemma~\ref{lem:Mean value},

For $n\geq 2$, the mean value of $\EN_\mu(A;h)$ over all $A\in
\mathcal M_n$ is :
\begin{equation}\label{mean mu}
\ave{\EN_\mu(\bullet;h)}  =0 \;.
%= \frac 1{q^n}\sum_{A\in \mathcal M_n}
%\EN_\mu(A;h)= \frac{q^{h+1}}{q^n}\sum_{f\in \EM_n} \mu(f)
%, \quad n>0
\end{equation}
%and hence for $n>1$,
%\begin{equation}
%\ave{\EN_\mu(\bullet;h)}=0
%\end{equation}
%since the generating function of M\"obius is
%$$\sum_f \mu(f)u^{\deg f}  = 1-qu$$
Indeed, by Lemma~\ref{lem:Mean value}
\begin{equation}
\ave{\EN_\mu(\bullet;h)} =  \frac {H}{q^n} \sum_{  f\in \mathcal M_n
} \mu(f) \;.
\end{equation}
Now as is well known and easy to see, for $n\geq 2$,
\begin{equation}
\sum_{ f\in \mathcal M_n } \mu(f) = 0 \;,
\end{equation}
%This can be seen for instance by computing the generating function
%\begin{equation}
%\sum_{ f } \mu(f) u^{\deg f} = \prod_{P} (1-u^{\deg P}) = 1-qu
%\end{equation}
hence we obtain \eqref{mean mu}.

%We will show that
%the mean value of $\EN_\mu(A;h)$ (as we average
%over $A\in \mathcal M_n$) is   $0$, and

We will show that the variance is
\begin{thm}
If $0\leq h\leq n-5$ then
\begin{equation}
\Var\EN_\mu(\bullet;h)\sim H, \quad q\to \infty
\end{equation}
\end{thm}
%This is consistent with \eqref{GC conj} if we
%write it as $H/\zeta_q(2)$ where $\zeta_q(2) = \sum_{f \, monic}
%\frac 1{||f||^2}$ (which tends to $1$ as $q\to \infty$) and
%$H=q^{h+1}$ is the number of monic polynomials in the short
%interval.

%\subsection{The variance}
We use the general formula of Lemma~\ref{covariance lemma} which
gives
\begin{equation}\label{gen form mu}
\Var (\EN_\mu(\bullet;h)) = \frac 1{q^{2(n-h-1)}} \sum_{\substack{\chi \bmod t^{n-h}\\
\chi\neq \chi_0 \mbox{ even}}} |\EM(n;\mu\chi)- \EM(n-1;\mu\chi)|^2
\end{equation}
where
\begin{equation}
\EM(n;\mu\chi) = \sum_{f\in \mathcal M_n} \mu(f)\chi(f)\;.
\end{equation}

We claim that
\begin{lem}
Suppose that $\chi$ is a primitive even character modulo $t^{n-h}$.
Then
\begin{equation}\label{formula for M_n}
\EM(n;\mu\chi) = \sum_{k=0}^n q^{k/2} \tr \Sym^k \Theta_\chi
\end{equation}
where $\Sym^n$
%: GL(N,\C)\to GL(\Sym^n \C^N)$
is the symmetric $n$-th power representation ($n=0$ corresponds to
the trivial representation).
In particular,
\begin{equation}
\EM(n;\mu\chi)- \EM(n-1;\mu\chi) =  q^{n/2}\tr \Sym^n \Theta_\chi
\;.
\end{equation}
If $\chi\neq \chi_0$ is not primitive, then
\begin{equation}\label{imprimitive M}
|\EM(n;\mu\chi)| \ll_n q^{n/2} \;.
\end{equation}
\end{lem}

\begin{proof}
We compute the generating function
\begin{equation}\label{gen M mu SI}
\sum_{n=0}^\infty \EM(n;\mu\chi) u^n = \sum_{f\; {\rm monic}} \chi(f)\mu(f) u^{\deg
f}  = \frac 1{L(u,\chi)}
\end{equation}
where $L(u,\chi) = \sum_{f\;{\rm monic}} \chi(f)u^{\deg f}$ is the associated
Dirichlet L-function. Now if $\chi$ is {\em primitive} and {\em
even}, then
\begin{equation}
L(u,\chi) = (1-u)\det (I-uq^{1/2}\Theta_\chi)
\end{equation}
where $\Theta_\chi\in U(n-h-2)$ is the unitarized Frobenius class.
Therefore we find
\begin{equation}
(1-u)\sum_{n=0}^\infty \EM(n;\mu\chi) u^n = \frac 1{\det
(I-uq^{1/2}\Theta_\chi)} = \sum_{k=0}^\infty q^{k/2}\tr \Sym^k
\Theta_\chi u^k  \;,
\end{equation}
where we have used the identity
\begin{equation}
\frac 1{\det (I-uA)} = \sum_{k=0}^\infty u^k \tr \Sym^k A \;.
\end{equation}
Comparing coefficients gives \eqref{formula for M_n}.

For non-primitive but non-trivial characters $\chi\neq \chi_0$, the
L-function still has the form $L(u, \chi) =
\prod_{j=1}^{n-h-1}(1-\alpha_j u)$ with all inverse roots
$|\alpha_j|\leq \sqrt{q}$, and hence we obtain \eqref{imprimitive
M}.
\end{proof}

We can now compute the variance using \eqref{gen form mu}. We start
by bounding the contribution of non-primitive characters, whose
number is $O(\frac 1q \Phi_{\rm ev}(t^{n-h}))=O(q^{n-h-2})$, and by
\eqref{imprimitive M} each contributes $O(q^{n})$ to the sum in
\eqref{gen form mu}, hence the total contribution of non-primitive
characters is bounded by $O_n(q^h)$. Consequently we find
\begin{equation}
\Var  \EN_\mu(\bullet;h) =  \frac {q^{h+1}}   {\Phi_{\rm
ev}(t^{n-h})}
\sum_{\substack{\chi \bmod t^{n-h}\\
\chi {\rm \; even \; and\;  primitive}}} |\tr \Sym^n \Theta_\chi |^2
+O(q^h) \;.
\end{equation}

Using Theorem~\ref{thm:KatzKR even}(i) we get, once we
 replace the projective group by the unitary group,
\begin{equation}
\lim_{q\to \infty}\frac{\Var(\EN_\mu(\bullet;h))} {q^{h+1}} =
\int_{U(n-h-2)} \left|\tr \Sym^n U \right|^2 dU \;.
\end{equation}
Note that by Schur-Weyl duality   (and Weyl's unitary trick),
  $\Sym^n$ is an {\em irreducible}
representation. Hence
\begin{equation}
\int_{U(n-h-2)} \left|\tr \Sym^n U \right|^2 dU=1
\end{equation}
and we conclude that $\Var(\EN_\mu(\bullet;h))\sim q^{h+1}=H$, as claimed.

%\newpage

\section{Variance of the M\"obius function in arithmetic
progressions}\label{Sec:mob AP}

We define
$$
\ES_{\mu,n,Q}(A) = \ES_\mu(A) = \sum_{\substack{f\in \EM_n\\
f=A\bmod Q}} \mu(f) \;.
$$
\begin{thm}
If $n\geq \deg Q\geq 2$ then the mean value of $\ES_\mu(A) $ tends
to $0$ as $q\to \infty$, and
\begin{equation}
\Var_Q (\ES_\mu ) \sim \frac{q^n}{\Phi(Q)}\int_{U(Q-1)}\Big|
\tr\Sym^n U\Big|^2 dU =\frac{q^n}{\Phi(Q)}\;.
\end{equation}
\end{thm}

The mean value over all residues coprime to $Q$ is
\begin{equation}
\ave{\ES_\mu} = \frac 1{\Phi(Q)} \sum_{\substack{f\in
\EM_n\\\gcd(f,Q)=1}} \mu(f)  = \frac 1{\Phi(Q)}\EM(n,\mu\chi_0)\;.
\end{equation}
To evaluate this quantity, we consider the generating function
\begin{equation}
\begin{split}
\sum_{n=0}^\infty \EM(n,\mu\chi_0)u^n &= \sum_{\gcd(f,Q)=1} \mu(f) u^{\deg f}\\
& = \prod_{P\nmid Q} (1-u^{\deg P})
 = \frac {1-qu}{\prod_{P\mid Q} (1-u^{\deg P})}\\
&= \frac {1-qu}{\prod_{k} (1-u^k)^{\lambda_k}}
\end{split}
\end{equation}
where $\lambda_k$ is the number of prime divisors of $Q$ of degree
$k$. Using the expansion
$$ \frac 1{(1-z)^\lambda} = \sum_{n=0}^\infty \binom{n+\lambda-1}{\lambda-1}z^n$$
gives
$$
 \frac {1}{\prod_{k} (1-u^k)^{\lambda_k}} = \sum_{n=0}^\infty C(n)u^n
$$
with
$$ C(n) = \sum_{\sum_k  k n_k=n} \prod_k  \binom{n_k+\lambda_k-1}{\lambda_k-1}$$
and hence for $n\geq 1$,
$$
\EM(n,\mu\chi_0) = C(n)-qC(n-1) \;.
$$
Thus we find that for $n\geq 1$,
\begin{equation}
\ave{\ES_\mu} = \frac {C(n)-qC(n-1)}{\Phi(Q)}
\end{equation}
and in particular for $\deg Q>1$,
\begin{equation}
|\ave{\ES_\mu}| \ll \frac {q}{|Q|} \to 0, \quad q\to \infty \;.
\end{equation}

For the variance we use \eqref{gen formula for varianceAP} which
gives
\begin{equation}\label{prelim var muQ}
\Var_Q( \ES_\mu)  = \frac 1{\Phi(Q)^2} \sum_{\chi\neq \chi_0}
|\EM(n;\mu\chi)|^2 \;.
\end{equation}
As in \eqref{gen M mu SI}, the generating function of
$\EM(n;\mu\chi)$ is $1/L(u,\chi)$. Now for $\chi$ odd and primitive,
$L(u,\chi) = \det (I-uq^{1/2}\Theta_\chi)$ with $\Theta_\chi\in
U(\deg Q-1)$ unitary. Hence for $\chi$ odd and primitive,
\begin{equation}
\EM(n;\mu\chi) = q^{n/2}\tr\Sym^n\Theta_\chi \;.
\end{equation}
For $\chi\neq \chi_0$ which is not odd and primitive, we can still
write $L(u,\chi) = \prod_{j=1}^{\deg Q-1}(1-\alpha_j u)$ with all
inverse roots $|\alpha_j|\leq \sqrt{q}$, and hence for $\chi\neq
\chi_0$ we have a bound
\begin{equation}\label{imprimitive muQ}
|\EM(n;\mu\chi)| \ll_n  q^{n/2} \;.
\end{equation}
The number of even characters is $\Phi_{\rm ev}(Q) = \Phi(Q)/(q-1)$
and the number of non-primitive characters is $O(\Phi(Q)/q)$, hence
the number of characters which are not odd and primitive is
$O(\Phi(Q)/q)$. Inserting the bound \eqref{imprimitive muQ} into
\eqref{prelim var muQ} shows that the contribution of such
characters is $O(q^{n-1}/\Phi(Q))$. Hence
\begin{equation}
\Var_Q \ES_\mu  = \frac{q^n}{\Phi(Q)} \frac 1{\Phi(Q)} \sum_{\chi
{\rm \; odd \; primitive}} |\tr\Sym^n\Theta_\chi|^2
+O(\frac{q^{n-1}}{\Phi(Q)})\;.
\end{equation}

Using Theorem~\ref{thm:KatzKR odd}(i) gives that as $q\to \infty$,
\begin{equation}
\Var_Q \ES_\mu  \sim \frac{q^n}{\Phi(Q)}\int_{U(Q-1)}\Big| \tr\Sym^n
U\Big|^2 dU =\frac{q^n}{\Phi(Q)}\;.
\end{equation}

%\newpage

\section{The variance of squarefrees in short intervals}
\label{sec:var sf si}

In this section we study the variance of the number of squarefree
polynomials in short intervals. The total number of squarefree
monic polynomials of degree $n>1$ is (exactly)
\begin{equation}
\sum_{f\in \mathcal M_n} \mu(f)^2 = \frac
{q^n}{\zeta_q(2)}=q^n(1-\frac 1q) \;.
\end{equation}
The number of squarefree polynomials in the short interval $I(A;h)$
is
\begin{equation}
\EN_{\mu^2}(A;h)= \sum_{f\in I(A;h)} \mu(f)^2 \;.
\end{equation}
%\begin{verbatim}
%we impose the condition f(0)\neq 0 which is unnatural here so will
%need to dispose of it later.
%\end{verbatim}

\begin{thm}
Let $0\leq h\leq n-6$. Assume $q\to \infty$ with all $q$'s coprime
to $6$.

i) If $h$ is even then
\begin{equation}
\Var{\EN_{\mu^2}(\bullet;h)} \sim q^{\frac h2}
\int\limits_{U(n-h-2)} \left| \tr \Sym^{\frac h2+1}U  \right|^2
dU=\frac{\sqrt{H}}{\sqrt{q}} \;.
\end{equation}
%(the matrix integral works out to be $1$).

ii) If $h$ is odd then
\begin{equation}
\begin{split}
\Var{\EN_{\mu^2}(\bullet;h)} &\sim q^{\frac{h-1}2}
\int\limits_{U(n-h-2)}|\tr U |^2dU  %\Lambda^{n-h-3}(U)|^2dU
\int\limits_{U(n-h-2)}|\tr\sym^{\frac{h+3}2}U'|^2dU' \\
&=\frac{\sqrt{H}}{q} \;.
\end{split}
\end{equation}
\end{thm}

\begin{proof}
To compute the variance, we use Lemma~\ref{covariance lemma}. Since
$\mu^2(t^m)=1$ for $m=0,1$ and equals $0$ for $m>1$, we obtain
\begin{equation}\label{formula for variance}
 \Var(\EN_{\mu^2}(\bullet;h)) =  \frac 1{\Phi_{ev}(t^{n-h})^2}
 \sum_{\substack{\chi\neq \chi_0 \bmod t^{n-h}\\ {\rm even}}}|\EM(n;\mu^2\chi) + \EM(n-1;\mu^2\chi)|^2
\end{equation}
where
\begin{equation}
\EM(n;\mu^2\chi) = \sum_{f\in \mathcal M_n} \mu(f)^2\chi(f) \;.
\end{equation}
%\begin{verbatim}
%consider changing variable \chi \mapsto \chi^2
%\end{verbatim}
To obtain an expression for $\EM(n;\mu^2\chi)$, we consider the
generating function
\begin{equation}
\sum_{n=0}^\infty  \EM(n;\mu^2\chi) u^n = \sum_f \mu(f)^2\chi(f)
u^{\deg f} = \frac{L(u,\chi)}{L(u^2,\chi^2)} \;.
\end{equation}
%Our task is to extract a tractable expression for the RHS.

%We saw that
%$$\sum_{n=0}^\infty \EM(n;\mu^2\chi) u^n = \frac{L(u,\chi)}{L(u^2,\chi^2)}$$
Assume that $\chi$ is primitive, and that $\chi^2$ is
also\footnote{If $q$ is odd then primitivity of $\chi$ and of
$\chi^2$ are equivalent. } primitive (modulo $t^{n-h}$).
 Then
 \begin{equation*}
 L(u,\chi) = (1-u)
\det(I-uq^{1/2}\Theta_\chi),\quad L(u^2,\chi^2) =
(1-u^2)\det(I-u^2q^{1/2}\Theta_{\chi^2}) \;.
\end{equation*}
Writing for $U\in U(N)$
\begin{equation}
\det(I-xU) = \sum_{j=0}^N \lambda_j(U) x^j, \quad \frac
1{\det(I-xU)} = \sum_{k=0}^\infty \tr\Sym^k U x^k
\end{equation}
gives, on abbreviating
$$\lambda_j(\chi):=\lambda_j( \Theta_\chi) , \quad
\sym^k(\chi^2) = \tr\sym^k \Theta_{\chi^2}
$$
that
\begin{equation*}
\begin{split}
\frac{L(u,\chi)}{L(u^2,\chi^2)}& = \frac{  \det(I-uq^{1/2}\Theta_\chi)}{(1+u) \det(I-u^2q^{1/2}\Theta_{\chi^2})}\\
&= \sum_{m=0}^\infty \sum_{0\leq j\leq N} \sum_{k=0}^\infty (-1)^{m}
\lambda_j(\chi)   \Sym^k (\chi^2) q^{(j+k)/2} u^{m+j+2k}
%\\&= \sum_{n=0}^\infty u^n(-1)^n \sum_{\substack{j+2k\leq n\\ 0\leq
%j\leq N\\ k\geq 0}} \lambda_j(\chi) \tr \Sym^k (\chi^2) q^{(j+k)/2}
\end{split}
\end{equation*}
and hence
\begin{equation}
\EM(n;\mu^2\chi) = (-1)^n  \sum_{\substack{j+2k\leq n\\ 0\leq j\leq
N\\ k\geq 0}} (-1)^j\lambda_j(\chi)   \Sym^k (\chi^2) q^{(j+k)/2}
\;.
\end{equation}
Therefore
\begin{equation*}
\begin{split}
\EM(n;\mu^2\chi) +\EM(n-1;\mu^2\chi)&= (-1)^n
\sum_{\substack{j+2k\leq  n\\ 0\leq j\leq N\\ k\geq 0}} (-1)^j
\lambda_j(\chi)   \Sym^k (\chi^2) q^{\frac{j+k}2}  \\
&+ (-1)^{n-1} \sum_{\substack{j+2k\leq n-1\\ 0\leq j\leq N\\ k\geq
0}} (-1)^j\lambda_j(\chi)   \Sym^k (\chi^2) q^{\frac{j+k}2}
\\
&=(-1)^n  \sum_{\substack{j+2k =n\\ 0\leq j\leq N\\ k\geq 0}}
(-1)^j\lambda_j(\chi)  \Sym^k (\chi^2) q^{\frac{j+k}2}\\
&= q^{n/4}\sum_{\substack{0\leq j\leq N\\j=n\bmod 2}}
\lambda_j(\chi) \sym^{\frac{n-j}2}(\chi^2) q^{j/4} \;.
\end{split}
\end{equation*}

Therefore, recalling that $N=n-h-2$,
\begin{multline*}
\EM(n;\mu^2\chi) +\EM(n-1;\mu^2\chi) \\= (-1)^n\begin{cases}
q^{\frac n2-\frac{h+1}4-\frac 14}
\lambda_N(\chi)\sym^{ \frac{h+2}2}(\chi^2),& n=N\bmod 2 \\
q^{\frac n2-\frac{h+1}4-\frac 12} \lambda_{N-1}(\chi)
\sym^{\frac{h+3}2}(\chi^2),& n\neq N\bmod 2
\end{cases}
  \times (1 + O(q^{-1/2}).
\end{multline*}
Noting that $n=N\bmod 2$ is equivalent to $h$ even, we finally
obtain
\begin{multline}\label{M(n)+M(n-1)}
|\EM(n;\mu^2\chi) +\EM(n-1;\mu^2\chi) |^2\\=
\begin{cases} q^{  n -\frac{h+1}2-\frac 12}
|\lambda_N(\chi)\sym^{ \frac{h+2}2}(\chi^2)|^2,& h\;{\rm even} \\
q^{n-\frac{h+1}2-1} |\lambda_{N-1}(\chi)
\sym^{\frac{h+3}2}(\chi^2)|^2,& h\;{\rm odd}
\end{cases}
  \times (1 + O(q^{-1/2})\;.
\end{multline}

%We wish to take the highest power of $q$. A computation shows that
%\begin{equation}
%\max_{\substack{ j+2k\leq n\\  0\leq j\leq N=n-(h+2)\\  k\geq
%0}}(j+k) = \begin{cases}
%n-\frac{h+2}2,& (j,k)=(N,\frac{n-N}2 = \frac h2+1), h \mbox{ even} \\
%n-\frac{h+3}2,& (j,k)=(N-1,\frac{n-N+1}2 = \frac{ h+3}2) \mbox{ or
%}\\ &(N,\frac{n-N-1}2=\frac{h+1}2), h \mbox{ odd}
%\end{cases}
%\end{equation}
%(recall $N=n-h-2$).

Inserting \eqref{M(n)+M(n-1)} into \eqref{formula for variance}
gives an expression for the variance, up to terms which are smaller
by $q^{-1/2}$. The contribution of non-primitive characters is
bounded as in previous sections and we skip this verification. We
separate cases according to $h$ even or odd.

\subsubsection{$h$ even}
%Assume first that $h$ is even. Then we find
%\begin{equation}
%\EM(n;\mu^2\chi) = (-1)^n \det(\Theta_\chi) \tr \Sym^{\frac
%h2+1}(\Theta_{\chi^2}) q^{\frac n2-\frac{(h+2)}4}(1+O(q^{-1/2}))
%\end{equation}
%and hence
%\begin{equation}
%|\EM(n;\mu^2\chi)|^2 = \frac{q^n}{q^{\frac{h+2}2}}\left | \tr
%\Sym^{\frac h2+1}(\Theta_{\chi^2})\right|^2 \Big(1+O(q^{-1/2}) \Big)
%\end{equation}

We have $|\lambda_N(\chi)|=|\det \Theta_\chi|=1$, so that
\begin{equation}
\Var{\EN_{\mu^2}(\bullet;h)} \sim q^{\frac h2} \frac
1{\Phi_{ev}(t^{n-h}) }
 \sum_{\substack{\chi \bmod t^{n-h}\\  {\rm primitive\; even}}}
|\sym^{ \frac{h+2}2}(\chi^2)|^2 \;.
\end{equation}
Here   change variables $\chi \mapsto \chi^2$, which is an
automorphism of the group of even characters if $q$ is odd, since
then the order of the group is $\Phi_{ev}(t^{n-h}) = q^{n-h-1}$ is
odd.
 Using Theorem~\ref{thm:KatzKR even}(i) for even primitive characters modulo $t^{n-h}$
 allows us to
replace the average over characters by a matrix integral, leading to
\begin{equation}
\Var{\EN_{\mu^2}(\bullet;h)} \sim q^{\frac h2}
\int\limits_{U(n-h-2)} \left| \tr \Sym^{\frac h2+1} U \right|^2 dU
\;.
\end{equation}
Since the symmetric powers $\Sym^k$ are irreducible representations,
the matrix integral works out to be $1$. Hence (with $H=q^{h+1}$)
\begin{equation}
\Var{\EN_{\mu^2}(\bullet;h)} \sim q^{  h/2} =
\frac{\sqrt{H}}{\sqrt{q}} \;.
\end{equation}

%\begin{verbatim}
%We need to argue that $\chi $ and $\chi^2$ non-primitive give a
%negligible contribution.
%\end{verbatim}

\subsubsection{$h$ odd}
Next, assume $h$ is odd. Then
%\begin{multline*}
%\EM(n;\mu^2\chi) = \left\{ \Sc_N(\Theta_\chi) \tr
%\Sym^{\frac{h+1}2}(\Theta_{\chi^2})   + \Sc_{N-1}(\Theta_\chi) \tr
%\Sym^{\frac{h+3}2}(\Theta_{\chi^2})
%\right\} \\
%\times (-1)^n q^{\frac n2-\frac{h+3}4} (1+O(q^{-1/2}))
%\end{multline*}
%and recalling that $\Sc_{N-1}(U) = \det(U) \overline{\tr U}$ gives
%\begin{equation}
%|\EM(n;\mu^2\chi)|^2 = \frac{q^n}{q^{\frac{h+3}2}} \left| \tr
%\Sym^{\frac{h+1}2}(\Theta_{\chi^2}) + \overline{\tr \Theta_\chi} \tr
%\Sym^{\frac{h+3}2}(\Theta_{\chi^2})\right|^2 (1+O(q^{-1/2}))
%\end{equation}
\begin{equation}
\Var{\EN_{\mu^2}(\bullet;h)} \sim q^{\frac {h-1}2} \frac
1{\Phi_{ev}(t^{n-h}) }
 \sum_{\substack{\chi \bmod t^{n-h}\\  {\rm primitive\; even}}}
|\lambda_{N-1}(\chi) \sym^{\frac{h+3}2}(\chi^2)|^2 \;.
\end{equation}
Note that $|\lambda_{N-1}(U)|=|\tr U|$, because $\lambda_{N-1}(U) =
(-1)^{N-1} \det U \tr U^{-1}$ and for unitary
matrices, $|\det U|=1$ and $\tr U^{-1}  = \overline{\tr U}$.

We now use Theorem~\ref{thm:KatzKR even}(ii), which asserts that, for
$0\leq h\leq n-6$ and $q\to \infty$ with $q$ coprime to $6$, both
$\Theta_\chi$ and $\Theta_{\chi^2}$ are uniformly distributed in
$PU(n-h-2)$   and that $\Theta_\chi$, $\Theta_{\chi^2}$ are {\em
independent}. We obtain
\begin{equation}
\begin{split}
\Var{\EN_{\mu^2}(\bullet;h)} &\sim
q^{\frac{h-1}2}\int\limits_{U(n-h-2)}|\tr U|^2dU
\int\limits_{U(n-h-2)}|\tr\sym^{\frac{h+3}2}U'|^2dU' \\
&=\frac{\sqrt{H}}{q}
\end{split}
\end{equation}
by irreducibility of the  symmetric power representations.
\end{proof}

\section{squarefrees  in arithmetic progressions}

As in previous sections, we set
\begin{equation}
\ES(A) = \sum_{\substack{f=A\bmod Q\\ f\in \mathcal M_n}}
\mu^2(f)\;.
\end{equation}
%For simplicity, we consider here the case of prime modulus. ??????
%\marginpar{squarefree}

We have the expected value
\begin{equation}
\ave{\ES}_Q = \frac 1{\Phi(Q)} \sum_{\substack{ f\in \mathcal M_n\\
(f,Q)=1}} \mu^2(f) \sim \frac{q^n/\zeta_q(2)}{\Phi(Q)}\sim
\frac{q^n}{|Q|}
\end{equation}
and the variance
\begin{equation}
\Var_Q(\ES) = \frac 1{\Phi(Q)^2} \sum_{\chi\neq \chi_0}
|\EM(n;\mu^2\chi)|^2\;.
\end{equation}

\begin{thm}
Fix $N\geq 1$. For any sequence of finite fields $\fq$, with $q$ odd, %$q=p^e$, of characteristic $p>2$,
and squarefree polynomials $Q\in \fq[t]$  with $\deg Q=N+1$, as
$q\to \infty$,
$$
\Var_Q(\ES)\sim  \frac {q^{n/2}}{|Q|^{1/2} } \times
\begin{cases}
 1/\sqrt{q}, & n\neq \deg Q\bmod 2 \\
 \\   1/q,  &  n= \deg Q\bmod 2\;.
\end{cases}
$$
\end{thm}
\begin{proof}
The generating function of $\EM(n;\mu^2\chi)$ is
\begin{equation}
\sum_{n=0}^\infty \EM(n;\mu^2\chi)u^n  = \sum_f
\mu(f)^2\chi(f)u^{\deg f} = \frac{L(u,\chi)}{L(u^2,\chi^2)}\;.
\end{equation}
If both $\chi$, $\chi^2$ are primitive, odd, characters (which
happens for almost all $\chi$), then
\begin{equation}
L(u,\chi) = \det (I-u q^{1/2}\Theta_\chi), \quad L(u^2,\chi^2) =
\det (I-u^2 q^{1/2}\Theta_{\chi^2})
\end{equation}
and writing (with $N=\deg Q-1$)
\begin{equation}
\det (I-u q^{1/2}\Theta_\chi) =\sum_{j=0}^{N} \lambda_j(\chi)
q^{j/2} u^j
\end{equation}
%where we have set
%\begin{equation}
%\lambda_j(\chi) = (-1)^j\Sc_j(\Theta_\chi)
%\end{equation}
%(recall $\det(I+xU) = \sum_{j=0}^N \Sc_j(U)x^j$). Moreover
\begin{equation}
\frac 1{\det(I-q^{1/2}u^2\Theta_{\chi^2})} = \sum_{k=0}^\infty
\sym^k(\chi^2) q^{k/2}u^{2k}
\end{equation}
%where
%\begin{equation}
%\sym^k(\chi^2) = \tr \Sym^k \Theta_{\chi^2}
%\end{equation}
we get, for $n\geq N$, (which is the interesting range)
\begin{equation}
\begin{split}
\EM(n;\mu^2\chi) &= \sum_{\substack{j+2k=n\\0\leq j \leq N\\k\geq
0}} \lambda_j(\chi)\sym^k(\chi^2) q^{\frac{j+k}2} \\
& = q^{n/4} \sum_{ \substack{ j=0\\j=n\bmod 2}}^N \lambda_{j}(\chi)
\sym^{\frac{n-j}2}(\chi^2) q^{j/4}
\end{split}
\end{equation}
and hence
\begin{multline}
\EM(n;\mu^2\chi) =\\(1+O(q^{-\frac 12}) ) \times q^{\frac{n+N}4}
\begin{cases} \lambda_N(\chi)
\sym^{\frac{n-N}2}(\chi^2), & n=N\bmod 2\\
  \\
  q^{-\frac
14}\lambda_{N-1}(\chi)\sym^{\frac{n-N+1}2}(\chi^2), & n\neq N\bmod
2\end{cases}
\end{multline}
Since $\lambda_N(\chi) = \det \Theta_\chi$, which has absolute value
one, we find
\begin{multline}
|\EM(n;\mu^2\chi)|^2 =\\  (1+O(q^{- \frac 12})) \times
q^{\frac{n+N}2}\begin{cases} |\sym^{\frac{n-N}2}(\chi^2)|^2, &
n=N\bmod 2\\
\\ q^{- \frac 12}|\lambda_{N-1}(\chi)\sym^{\frac{n-N+1}2}(\chi^2)|^2, & n\neq
N\bmod 2\;.\end{cases}
\end{multline}

If $\chi\neq \chi_0$ is not odd, or not primitive, we may use the
same computation to show that
\begin{equation}
|\EM(n;\mu^2\chi)| \ll_n q^{(n+N)/4},\quad \chi\neq \chi_0\;{\rm
even\; or \; imprimitive}.
\end{equation}
We thus have a formula for $\Var(\ES)$.  We may neglect the
contribution of characters $\chi$ for which $\chi$ or $\chi^2$ are
non-primitive, or even as these form a proportion $\leq 1/q$ of all
characters, and thus their contribution is
$$
\ll \frac 1{\Phi(Q)}\frac 1q q^{(n+N)/2} \ll \frac 1q \frac
{q^{n/2}}{|Q|^{1/2}\sqrt{q}}
$$
which is negligible relative to the claimed main term in the
Theorem.

To handle the contribution of primitive odd characters we invoke
Theorem~\ref{thm:KatzKR odd}(ii) which asserts that  both $\Theta_\chi$
and $\Theta_{\chi^2}$ are uniformly distributed in $U(\deg Q-1)$ and
are independent (for $q$ odd). To specify the implications, we
separate into cases:

If $n=N \bmod 2$ (i.e. $n\neq \deg Q\bmod 2$) then
\begin{equation}
 \Var_Q(\ES) \sim \frac {q^{n/2}}{|Q|^{1/2}q^{1/2}} \cdot \frac 1{\Phi(Q)}
\sum_{\chi,\chi^2 \mbox{ primitive}}|\sym^{\frac{n-N}2}(\chi^2)|^2
\;.
\end{equation}
By equidistribution
\begin{equation}
\frac 1{\Phi(Q)} \sum_{\chi,\chi^2 \mbox{
primitive}}|\sym^{\frac{n-N}2}(\chi^2)|^2 \sim  \int_{U(N)} \left|
\tr \Sym^{\frac{n-N}2} U\right|^2 dU\;.
\end{equation}
Note that $\int_{U(N)} \left| \tr \Sym^{\frac{n-N}2} U\right|^2
dU=1$ by irreducibility of $\Sym^k$. Thus we obtain
\begin{equation}
\Var_Q(\ES)\sim \frac {q^{n/2}}{|Q|^{1/2}q^{1/2}}, \quad n\neq \deg
Q\bmod 2\;.
\end{equation}
%A  problem with this reasoning is that the map $\chi\mapsto \chi^2$
%on Dirichlet characters mod $Q$  is neither one-to-one, nor is it
%surjective if $q$ is odd. The size of its kernel and co-kernel have
%to do with the highest power of $2$ dividing $\Phi(Q) = q^{\deg
%Q}-1$ (recall $Q$ is assumed prime), and if the map is not
%surjective on the set of odd characters we do not know how to use
%equidistribution.

%At this point we only have an upper bound
%\begin{equation}\label{upper bd on var  even case}
%\Var (\ES) \ll_{n,\deg Q} \frac {q^{n/2}}{|Q|^{1/2}q^{1/2}}
%\end{equation}

If $n\neq N\bmod 2$ (i.e. $n=\deg Q\bmod 2$), then we get
\begin{equation}
\Var_Q(\ES)\sim \frac {q^{n/2}}{q|Q|^{1/2}} \frac 1{\Phi(Q)}
\sum_{\chi,\chi^2 \mbox{ primitive}}
\left|\lambda_{N-1}(\chi)\sym^{\frac{n-N+1}2}(\chi^2)\right|^2\;.
\end{equation}
Note that as in \S~\ref{sec:var sf si}, $|\lambda_{N-1}(\chi)| =
|\tr \Theta_\chi|$. By Theorem~\ref{thm:KatzKR odd}(ii),
\begin{multline}
 \frac 1{\Phi(Q)}\sum_{\chi,\chi^2 \mbox{ primitive}}
\left|\lambda_{N-1}(\chi)\sym^{\frac{n-N+1}2}(\chi^2)\right|^2 \sim
\\
\iint_{U(N)\times U(N)} \left|\tr U\right|^2 \cdot
\left|\tr\sym^{\frac{n-N+1}2}U' \right|^2dUdU' = 1
\end{multline}
and hence
\begin{equation}
\Var_Q(\ES)\sim \frac {q^{n/2}}{q|Q|^{1/2}} \;.
\end{equation}

%\begin{verbatim} At this point we do not know how to use
%equidistribution to evaluate the chi average
%\end{verbatim}

%In particular we get an upper bound
%\begin{equation}
%\Var(\ES)\ll_{n,\deg Q} \frac{q^{n/2}}{q|Q|^{1/2}} ,\quad n=\deg Q
%\bmod 2
%\end{equation}
%which is of smaller order than \eqref{upper bd on var even case} for
%the case $n=N\bmod 2$.
%
%Assuming that $\Theta_\chi$, $\Theta_{\chi^2}$ are both uniformly
%distributed in $U(\deg Q-1)$ and {\em independent} would give
Thus we find
\begin{equation}
\Var_Q(\ES)\sim \begin{cases} \frac {q^{n/2}}{|Q|^{1/2}q^{1/2}},
\quad n\neq \deg Q\bmod 2 \\    \\
\frac {q^{n/2}}{|Q|^{1/2}q  }, \quad n= \deg Q\bmod 2
\end{cases}
\end{equation}
as claimed.
\end{proof}

%\newpage

\appendix
\section{Hall's theorem for $\fq[t]$: The large degree limit}
\label{appendix}

Let $Q(n,H)$ be the number of squarefree integers in an interval of
length $H$ about $n$:
\begin{equation}
Q(n,H):=%\#\{N-\frac 12 H<n\leq n+\frac 12 H: n \mbox{ squarefree}\}=
\sum_{j=1}^H \mu^2(n+j)\;.
\end{equation}
Hall \cite{Hall} studied the variance of $Q(n,H)$ as $n$ varies up
to $X$. He showed that provided $H=O(X^{2/9-o(1)})$, the variance
grows like $\sqrt{H}$ and in fact
%of $Q(n,H)$
admits an asymptotic formula:
\begin{equation}\label{Hall thm}
\frac 1X \sum_{n\leq X} \left |Q(n,H) - \frac H{\zeta(2)}\right |^2
\sim A\sqrt{H}
\end{equation}
with
\begin{equation}\label{Hall constant}
A = \frac{\zeta(3/2)}\pi \prod_p(1-  \frac3{p^2}+\frac{2}{p^3})\;.
\end{equation}

We give a version of Hall's theorem for the polynomial ring $\fq[t]$
with $q$ fixed. Let
$$
\EN(A) = \sum_{|f-A|\leq q^h} \mu^2(f)
$$
be the number of squarefree polynomials in a short interval $I(A;h)$
around $A\in \EM_n$, with $h\leq n-2$. Note that
$$\#I(A;h) = q^{h+1}=:H \;.$$
We wish to compute the variance of $\EN$ as we average over all
short intervals with $q$ fixed and $n\to \infty$.

Let
\begin{equation}
\beta_q =\prod_P(1-\frac3{|P|^2}+\frac{2}{|P|^3})\;.
\end{equation}
Our result is
%assuming  certain  Hardy-Littlewood heuristics are  valid, then
\begin{thm}\label{Thm Hall}
As $h\to \infty$,
%\fbox{HOW DOES $h$ depend on $n$ ???}
\begin{equation*}
%\begin{split}
\Var \EN = \sqrt{H}
 \frac{\beta_q}{1-\frac 1{q^3}}
\begin{cases}
\frac {1+\frac 1{q^2}}{\sqrt{q}},& h\;{\rm even}\\    \\
\frac{1+\frac 1q}{q},&h\;{\rm odd}
\end{cases}    + O(\frac{H^2n}{q^{n/3}}) +O_q(H^{1/4+o(1)}) \;.
%\\
%& = q^h (\frac 1{\zeta(2)} -\beta  \frac{ 1+\frac 1q }{1-\frac
%1{q^3}}) + \frac{\beta}{1-\frac 1{q^3}} q^{(h+1)/2}
%\begin{cases}
%\frac {1+\frac 1{q^2}}{\sqrt{q}},& h\;{\rm even}\\    \\
%\frac{1+\frac 1q}{q},&h\;{\rm odd}
%\end{cases}
%\end{split}
   \end{equation*}
\end{thm}
In particular we get an asymptotic result provided $h<(\frac
29-o(1))n$. i.e. $H<(q^n)^{\frac 29-o(1)}$. It is likely that one
can improve the factor $2/9$ a bit.

\subsection{The probability that $f$ and $f+J$ are both squarefree}
As in the number field case, we start with an expression for the
probability that both $f$ and $f+J$ are squarefree. For a non-zero
polynomial $J\in \fq[t]$, define the ``singular series"
 \begin{equation}\label{HL S(J)}
\mathfrak S(J) = \prod_P(1-\frac 2{|P|^2}) \cdot \prod_{P^2\mid
J}\frac{|P|^2-1}{|P|^2-2}\;,
\end{equation}
the product over all prime polynomials.
 We will first show
\begin{thm}\label{thm hl}
 For $0\neq J\in \fq[t]$, $\deg J<n$,
\begin{equation}
S(J;n):=\sum_{f\in \EM_n} \mu^2(f)\mu^2(f+J)  = \mathfrak S(J) q^n +
O(nq^{\frac {2n}3})\;,
\end{equation}
the implied constant absolute, with $\mathfrak S(J)$ given by
\eqref{HL S(J)}.
\end{thm}
Note that Theorem~\ref{thm hl} is uniform in $J$ as long as $\deg
J<n$.

Theorem~\ref{thm hl} is the exact counterpart for the analogous
quantity over the integers, which has been known in various forms
since the 1930's. The proof below is roughly the same as the one
given in  \cite[Theorem 1]{Hall}. The  exponent $2/3$ has been
improved, by Heath Brown \cite{Heath Brown} to $7/11$ and by Reuss
\cite{Reuss} to about $0.578\ldots$.

\subsection{A decomposition of $\mu^2$}
We start with the identity
\begin{equation}
\mu^2(f) = \sum_{d^2\mid f} \mu(d)
\end{equation}
(the sum over monic $d$). Pick an integer  parameter $0<z\leq n/2$,
write $Z=q^z$, and decompose the sum into two parts, one over
"small" divisors, that is with $\deg d<z$, and one over "large"
divisors:
\begin{equation}
\mu^2 = \mu^2_z+e_z
\end{equation}
\begin{equation}
\mu^2_z(f) = \sum_{\substack{d^2\mid f\\ \deg d<z}}\mu(d),\quad
e_z(f)= \sum_{\substack{d^2\mid f\\ \deg d\geq z}}\mu(d) \;.
\end{equation}
Let
\begin{equation}
S_z(J;n):=\sum_{f\in \EM_n} \mu^2_z(f )\mu^2_z(f+J)\;.
\end{equation}
We want to replace $S$ by $S_z$.

\subsection{Bounding $S(J;n)-S_z(n;J)$}
\begin{prop}\label{prop S-Sz}
If $z\leq n/2$ then
$$
\Big| S(J;n)-S_z(J;n)\Big| \ll \frac{q^n}{Z}\;,
$$
where $Z=q^z$.
\end{prop}
\begin{proof}
Note that
\begin{multline}
 \mu^2(f)\mu^2(f+J)=\mu^2_z(f)\mu^2_z(f+J) \\
   +
 e_z(f)\mu^2(f+J)-\mu^2(f)e_z(f+J) - e_z(f)e_z(f+J)
\end{multline}
so that (recall $\mu^2(f)\leq 1$)
\begin{multline}\label{ineq for mu2}
%\begin{split}
\Big|\mu^2(f)\mu^2(f+J)-\mu^2_z(f)\mu^2_z(f+J) \Big|  \leq |e_z(f)|
+ |e_z(f+J)| + |e_z(f)e_z(f+J)|
\\
 \leq |e_z(f)| + |e_z(f+J)|  + \frac 12|e_z(f)|^2+ \frac
12|e_z(f+J)|^2
%\end{split}
\end{multline}
 and therefore, summing \eqref{ineq for mu2} over $f\in \EM_n$ and noting
that since $\deg J<n$, sums of $f+J$ are the same as sums of $f$,
\begin{equation}\label{Bd for diff  e1 e2}
%\begin{split}
\Big| S(J;n)-S_z(J;n)\Big|
%&=\Big|  \sum_{f\in \EM_n}e_z(f)\mu^2(f+J) -\mu^2(f)e_z(f+J) - e_z(f)e_z(f+J)\Big|
%\\&
\leq  2\sum_{f\in \EM_n} |e_z(f)| + \sum_{f\in\EM_n}|e_z(f)|^2 \;.
%\end{split}
\end{equation}

We have
$$ |e_z(f)| = |\sum_{\substack{d^2\mid f\\ \deg d\geq z}} \mu(d)|
\leq \sum_{\substack{d^2\mid f\\ \deg d\geq z}}1
$$
so that
\begin{equation}\label{bound e1}
\begin{split}
\sum_{f\in \EM_n} |e_z(f)| &\leq \sum_{f\in
\EM_n}\sum_{\substack{d^2\mid f\\ \deg d\geq z}}1
\\
& = \sum_{z\leq \deg d\leq n/2} \#\{f\in \EM_n:d^2\mid f\}
\\
& =\sum_{z\leq \deg d\leq n/2}\frac{q^n}{|d|^2} \leq \frac{2q^n}{Z}
\;.
\end{split}
\end{equation}

Moreover,
\begin{equation*}
\begin{split}
\sum_{f\in\EM_n}|e_z(f)|^2  &\leq \sum_{f\in\EM_n}
\sum_{\substack{d_1^2\mid f\\ \deg d_1\geq z}}
\sum_{\substack{d_2^2\mid f\\ \deg d_2\geq z}} 1
\\
&\leq \sum_{z\leq \deg d_1, \deg d_2\leq n/2}\#\{f\in \EM_n:
d_1^2\mid f \;{\rm and }\; d_2^2\mid f\} \;.
\end{split}
\end{equation*}
Now the conditions $d_1^2\mid f$ and $d_2^2\mid f$ are equivalent to
$[d_1,d_2]^2\mid f$, where $[d_1,d_2]$ is the least common multiple
of $d_1$ and $d_2$, and this can only happen if $\deg[d_1,d_2]\leq
\deg f/2=n/2$, in which case the number of such $f$ is
$q^n/|[d_1,d_2]|^2$ and is zero otherwise. Thus
\begin{equation}\label{sum e^2}
\begin{split}
\sum_{f\in\EM_n}|e_z(f)|^2  &\leq
 \sum_{\substack{z\leq \deg d_1, \deg d_2\leq n/2\\\deg [d_1,d_2]\leq n/2}}
 \frac{q^n}{|[d_1,d_2]|^2}
 \\
 &\leq q^n \sum_{\deg d_1,\deg d_2\geq z} \frac 1{|[d_1,d_2]|^2}\;.
\end{split}
\end{equation}

We claim that (this is the analogue of \cite[Lemma 2]{Hall})
\begin{lem}\label{double sum d}
\begin{equation*}
\sum_{\deg d_1,\deg d_2\geq z} \frac 1{|[d_1,d_2]|^2} \ll \frac 1{Z}
\end{equation*}
\end{lem}
Inserting Lemma~\ref{double sum d} in \eqref{sum e^2} we will get
\begin{equation}\label{bound e2}
\sum_{f\in\EM_n}|e_z(f)|^2 \ll \frac {q^n}{Z} \;.
\end{equation}
Inserting \eqref{bound e1} and \eqref{bound e2} in \eqref{Bd for
diff  e1 e2} we conclude Proposition~\ref{prop S-Sz}.

To prove Lemma~\ref{double sum d},
%we follow Gerald Tenenbaum's proof of \cite[Lemma 2]{Hall}. For a monic polynomial $m$, set
%$$ f_z(m) = \#\{\deg d_1,\deg d_2\geq z: [d_1,d_2]=m\}$$
use $[d_1,d_2] = d_1d_2 /\gcd(d_1,d_2)$ to rewrite the sum as
\begin{equation*}
\begin{split}
\sum_{\deg d_1,\deg d_2\geq z} \frac 1{|[d_1,d_2]|^2}&= \sum_{\deg
d_1,\deg d_2\geq z} \frac{|\gcd(d_1,d_2)|^2}{|d_1|^2|d_2|^2}
\\
& = \sum_{k\;{\rm monic}} |k|^2 \sum_{\substack{\deg d_1,\deg
d_2\geq z\\ \gcd(d_1,d_2)=k}} \frac{1}{|d_1|^2|d_2|^2} \;.
\end{split}
\end{equation*}

In the sum above, we write $d_j=k\delta_j$ with
$\gcd(\delta_1,\delta_2)=1$. The condition $\deg d_j \geq z$
gives no restriction on $\delta_j$  if $\deg k\geq z$, and otherwise
translates into $\deg \delta_j \geq z-\deg k$. Thus
\begin{equation*}
\begin{split}
\sum_{\deg d_1,\deg d_2\geq z} \frac 1{|[d_1,d_2]|^2}& \ll
 \sum_{k\;{\rm monic}} \frac 1{|k|^2 }\sum_{\substack{\deg
\delta_1,\deg \delta_2\geq z-\deg k\\ \gcd(\delta_1,\delta_2)=1}}
\frac{1}{|\delta_1|^2|\delta_2|^2}
\\
&\leq \sum_{k\;{\rm monic}} \frac 1{|k|^2 }\Big(\sum_{\deg
\delta\geq z-\deg k} \frac 1{|\delta|^2} \Big)^2
\end{split}
\end{equation*}
after ignoring the coprimality condition.
 Therefore
\begin{equation*}
\begin{split}
\sum_{k\;{\rm monic}} \frac 1{|k|^2 }\Big(\sum_{\deg \delta\geq
z-\deg k} \frac 1{|\delta|^2} \Big)^2 &\leq \sum_{\deg k\leq z}
\frac 1{|k|^2}\Big(\sum_{\deg \delta \geq z-\deg k} \frac
1{|\delta|^2} \Big)^2
\\
&+ \sum_{\deg k>z} \frac 1{|k|^2}\Big(\sum_{\delta} \frac
1{|\delta|^2} \Big)^2
\\
&\ll  \sum_{\deg k\leq z} \frac 1{|k|^2}(\frac {|k |}{q^z})^2 +
\frac 1{q^{z+1}}
\\
&\ll \frac 1{Z} \;,
\end{split}
\end{equation*}
which proves Lemma~\ref{double sum d}.
\end{proof}

\subsection{Evaluating $S_z(J;n)$}
\begin{prop}\label{prop Sz}
If $z\leq n/2$ then
$$ S_z(J;n) = q^n \mathfrak S(J)  + O(\frac{q^nz}{Z}) + O(Z^2),$$
with $Z=q^z$.
\end{prop}
\begin{proof}
Using the definition of $\mu^2_z$, we obtain
\begin{equation*}
\begin{split}
S_z(J;n)&:= \sum_{f\in \EM_n} \mu^2_z(f)\mu^2_z(f+J)
\\
& = \sum_{\deg d_1\leq z}\sum_{\deg d_2\leq z} \mu(d_1)\mu(d_2)
\#\{f\in \EM_n: d_1^2\mid f, d_2^2 \mid f+J\} \;.
\end{split}
\end{equation*}
Decomposing into residue classes modulo  $[d_1,d_2]^2$ gives
\begin{multline*}
\#\{f\in \EM_n: d_1^2\mid f, d_2^2 \mid f+J\} \\
= \sum_{\substack{c\bmod [d_1,d_2]^2\\ c=0\bmod d_1^2\\ c=-J\bmod
d_2^2}} \#\{f\in \EM_n: f=c\bmod [d_1,d_2]^2\} \;.
\end{multline*}

%Since $\deg [d_1,d_2]^2\leq \deg(d_1^2d_2^2)\leq 4z\leq n$ by our
%choice of $z$, we have
If $\deg [d_1,d_2]^2\leq n$ then
$$\#\{f\in \EM_n: f=c\bmod [d_1,d_2]^2\} = \frac{q^n}{|[d_1,d_2]|^2}
\;.
$$
Otherwise there is at most {\em one} $f\in \EM_n$ with $f=c\bmod
[d_1,d_2]^2$. So we write
$$\#\{f\in \EM_n: f=c\bmod [d_1,d_2]^2\} = \frac{q^n}{|[d_1,d_2]|^2}
 + O(1)
\;.
$$

 Let $\kappa(d_1,d_2;J)$ be the number of solutions $c\bmod
[d_1,d_2]^2$ of the system of congruences $c=0\bmod d_1^2$,
$c=-J\bmod d_2^2$; it is either $1$ or $0$ depending on whether
$\gcd(d_1,d_2)^2\mid J$ or not. Then we have found that
\begin{equation*}
\begin{split}
S_z(J;n) &= \sum_{\deg d_1\leq z}\sum_{\deg d_2\leq z}
\mu(d_1)\mu(d_2)\kappa(d_1,d_2;J)
\Big( \frac{q^n}{|[d_1,d_2]|^2}  +O(1)\Big) \\
&= q^n\sum_{\deg d_1\leq z}\sum_{\deg d_2\leq z} \mu(d_1)\mu(d_2)
\frac{\kappa(d_1,d_2;J)}{|[d_1,d_2]|^2} + O(Z^2) \;.
\end{split}
\end{equation*}

The double sum can be extended to include all $d_1,d_2$:
\begin{multline*}
\sum_{\deg d_1\leq z}\sum_{\deg d_2\leq z} \mu(d_1)\mu(d_2)
\frac{\kappa(d_1,d_2;J)}{|[d_1,d_2]|^2} =
\sum_{d_1,d_2}\mu(d_1)\mu(d_2)
\frac{\kappa(d_1,d_2;J)}{|[d_1,d_2]|^2} \\+ O\Big(\sum_{\deg d_1>z}
\sum_{d_2} \frac 1{|[d_1,d_2]|^2} \Big) \;,
\end{multline*}
so that
\begin{multline}\label{interim Sz}
S_z(J;n)= q^n\sum_{d_1,d_2}\mu(d_1)\mu(d_2)
\frac{\kappa(d_1,d_2;J)}{|[d_1,d_2]|^2} \\+ O\Big(\sum_{\deg d_1>z}
\sum_{d_2} \frac 1{|[d_1,d_2]|^2} \Big) +O(Z^2) \;.
\end{multline}

We bound the sum in the remainder term of \eqref{interim Sz}  by
(this is the analogue of \cite[Lemma 3]{Hall}):
\begin{lem}\label{second double sum d}
\begin{equation*}
\sum_{\deg d_1>z} \sum_{d_2} \frac 1{|[d_1,d_2]|^2} \ll
\frac{z}{q^z} = \frac zZ \;.
\end{equation*}
\end{lem}
\begin{proof}
We argue as in the proof of Lemma~\ref{double sum d}: We write the
least common multiple as $[d_1,d_2] = d_1d_2/\gcd(d_1,d_2)$ and sum
over all pairs of $d_1,d_2$ with given $\gcd$:
\begin{equation*}
\begin{split}
\sum_{\deg d_1>z} \sum_{d_2} \frac 1{|[d_1,d_2]|^2}& = \sum_k |k|^2
\sum_{\deg d_1>z} \sum_{\substack{d_2\\ \gcd(d_1,d_2)=k}} \frac
1{|d_1|^2|d_2|^2}
\\
& = \sum_k |k|^2 \sum_{\deg \delta_1 >z-\deg k}\frac
1{|k|^2|\delta_1|^2}
\sum_{\substack{\delta_2\\\gcd(\delta_1,\delta_2)=1}} \frac
1{|k|^2|\delta_2|^2} \;,
\end{split}
\end{equation*}
after writing $d_j=k\delta_j$ with $\delta_1$, $\delta_2$ coprime.

Ignoring the coprimality condition gives
\begin{equation*}
\begin{split}
\sum_{\deg d_1>z} \sum_{d_2} \frac 1{|[d_1,d_2]|^2}& \ll \sum_k
\frac 1{|k|^2} \sum_{\deg \delta_1>z-\deg k} \frac
1{|\delta_1|^2}\sum_{\delta_2} \frac 1{|\delta_2|^2}
\\
&\ll \sum_{\deg k \leq z} \frac 1{|k|^2}\sum_{\deg \delta_1>z-\deg
k} \frac 1{|\delta_1|^2}+ \sum_{\deg k>z} \frac
1{|k|^2}\sum_{\delta_1} \frac 1{|\delta_1|^2}
\\
&\ll\sum_{\deg k \leq z} \frac 1{|k|^2} \frac{|k|}{q^z} +\sum_{\deg
k>z} \frac 1{|k|^2}
\\&\ll \frac{z}{q^z} = \frac{z}{Z} \;,
\end{split}
\end{equation*}
which proves Lemma~\ref{second double sum d}.
\end{proof}

Putting together \eqref{interim Sz} and Lemma~\ref{second double sum
d}, we have shown that
\begin{equation}\label{app:infinite sum}
S_z(J;n) = q^n\sum_{  d_1 }\sum_{  d_2 } \mu(d_1)\mu(d_2)
\frac{\kappa(d_1,d_2;J)}{|[d_1,d_2]|^2}  + O(\frac {q^nz}{Z})
+O(Z^2) \;.
\end{equation}

It remains to show that the infinite sum in \eqref{app:infinite sum}
coincides with the singular series $\mathfrak S(J)$:
\begin{lem}
\begin{equation*}
 \sum_{ d_1 }\sum_{  d_2 } \mu(d_1)\mu(d_2)
\frac{\kappa(d_1,d_2;J)}{|[d_1,d_2]|^2} = \mathfrak S(J)
\end{equation*}
\end{lem}
\begin{proof}
This is done exactly as in \cite[Appendix]{Hall}. We write
\begin{equation*}
\sum_{ d_1 }\sum_{  d_2 } \mu(d_1)\mu(d_2)
\frac{\kappa(d_1,d_2;J)}{|[d_1,d_2]|^2}  = \sum_m
\frac{s(m;J)}{|m|^2}
\end{equation*}
where
\begin{equation*}
s(m;J) = \sum_{\substack{ [d_1,d_2]=m\\ \gcd(d_1,d_2)^2\mid J}}
\mu(d_1)\mu(d_2)
\end{equation*}

One checks that $s(m;J)$ is multiplicative in $m$, and that for $P$
prime
$$s(P^\alpha;P^j) = \sum_{\substack{\max(u,v)=\alpha\\
2\min(u,v)\leq j}} \mu(P^u)\mu(P^v)
$$
so that $s(P^\alpha;P^j)=0$ for $\alpha\geq 2$ while for $\alpha=1$
\begin{equation*}
s(P;P^j) =\sum_{\substack{\max(u,v)=1\\ \min(u,v)\leq j/2}}
\mu(P^u)\mu(P^v)
\end{equation*}
If $j<2$ (that is if $P^2\nmid P^j$), then the sum is over
$\max(u,v)=1$ and $\min(u,v)=0$ i.e. $(u,v) = (0,1),(1,0)$ which
works out to $s(P,P^j)=-2$ for $j=0,1$, while for $j\geq 2$ the only
restriction is $\max(u,v)=1$, i.e. $(u,v) = (1,0),(0,1),(1,1)$ which
gives $s(P,P^j)=-1$ for $j\geq 2$. Thus
\begin{equation*}
\begin{split}
\sum_m \frac{s(m;J)}{|m|^2}  &= \prod_P \left(
1+\frac{s(P,J)}{|P|^2}
\right)\\
&= \prod_{P^2\mid J}(1-\frac 1{|P|^2})\prod_{P^2\nmid J} (1-\frac
2{|P|^2})
\end{split}
\end{equation*}
which is exactly $\mathfrak S(J)$.
\end{proof}
We now conclude the proof of Proposition~\ref{prop Sz}: By
Propositions~\ref{prop S-Sz}, \ref{prop Sz} we have shown that for
$z\leq n/2$,
$$
S(J;n) = q^n\mathfrak S(J) + O(\frac{q^nz}{Z})+O(Z^2)
$$
Taking $z\approx n/3$ gives that for all $J\neq 0$ with $\deg J<n$,
$$
S(J;n) = q^n\mathfrak S(J) + O(nq^{2n/3})
$$
as claimed.
\end{proof}

\subsection{Computing the variance}

As described in \S~\ref{sec:backgd on si}, we have a partition of the set $\EM_n$ of monic polynomials of degree $n$ as
$$ \EM_n = \coprod_{A\in \mathcal A}I(A;h)$$
where
$$
 \mathcal A = \{A=t^n+a_{n-1}t^{n_1}+\dots +a_{h+1}t^{h+1}: a_j\in
 \fq\}\;.
$$

The mean value of $\EN$ is, for $n\geq 2$,
$$
\ave{\EN} = \frac 1{\#\mathcal A}\sum_{A\in \mathcal A} \EN(A) =
\frac{q^{h+1}}{\zeta(2)}, \quad n\geq 2\;.
$$
The variance is
\begin{equation}
\Var{\EN} = \ave{\EN^2}-\ave{\EN}^2 \;.
\end{equation}
We have
\begin{equation*}
\begin{split}
\ave{\EN^2} &= \frac 1{\#\mathcal A}\sum_{A\in \mathcal A} \sum_{|f-A|\leq q^h} \sum_{|g-A|\leq q^h} \mu^2(f)\mu^2(g) \\
& =  \frac 1{\#\mathcal A}\sum_{f\in \EM_n} \mu^2(f) +
 \frac 1{\#\mathcal A} \sum_{\substack{f\neq g\\ |f-g|\leq q^h}} \mu^2(f)\mu^2(g)\\
&=\ave{\EN} + q^{h+1}\sum_{0\neq J\in \mathcal P_{\leq h} } \frac
1{q^n}\sum_{f\in \EM_n} \mu^2(f)\mu^2(f+J)\;.
\end{split}
\end{equation*}

We use Theorem~\ref{thm hl} %the Hardy-Littlewood heuristic
\begin{equation}
 \sum_{f\in \EM_n} \mu^2(f)\mu^2(f+J) = q^n\mathfrak S(J)  +
 O\Big(nq^{2n/3}\Big)
 \end{equation}
 where
 \begin{equation}\label{HL heuristic J}
\mathfrak S(J) = \prod_P(1-\frac 2{|P|^2}) \cdot \prod_{P^2\mid
J}\frac{|P|^2-1}{|P|^2-2}
 = \alpha \mathfrak s(J)
% \alpha \cdot \prod_{P^2\mid J}\frac{|P|^2-1}{|P|^2-2}
\end{equation}
with
\begin{equation}
\alpha  =\prod_P(1-\frac 2{|P|^2})
\end{equation}
and
$$
%\mathfrak S(J) = \alpha \mathfrak s(J), \quad
\mathfrak s(J) = \prod_{P^2\mid J} \frac{|P|^2-1}{|P|^2-2}\;.
$$
This gives that
\begin{equation}\label{var nHL}
\begin{split}
\Var & = \ave{\EN} - \ave{\EN}^2 + \alpha  q^{h+1} \sum_{0\neq J\in \mathcal P_{\leq h}} \mathfrak s(J)  + O(H^2nq^{-n/3})\\
& =\frac{q^{h+1}}{\zeta(2)} - ( \frac{q^{h+1}}{\zeta(2)})^2 + \alpha
q^{h+1} (q-1) \sum_{j=0}^{h} \sum_{J\in \EM_j} \mathfrak s(J)+
O(H^2nq^{-n/3})\;,
\end{split}
\end{equation}
the last step using  homogeneity: $\mathfrak S(cJ) = \mathfrak
S(J)$, $c\in \fq^\times$.

\subsection{Computing $\sum_J \mathfrak s(J)$}
To evaluate the sum of $\mathfrak s(J)$ in \eqref{var nHL}, we form
the generating series
$$
F(u) = \sum_{J \;{\rm monic}} \mathfrak s(J)u^{\deg J}\;.
$$
Since $\mathfrak s(J)$ is multiplicative,  and $\mathfrak s(P^k) =
1$ if $k=0,1$, and $\mathfrak s(P^k) = \mathfrak s(P^2) =
\frac{|P|^2-1}{|P|^2-2}$ if $k\geq 2$, we find
\begin{equation*}
\begin{split}
F(u) &= \prod_P (1 + u^{\deg P} + \mathfrak s(P^2) \sum_{k\geq 2} u^{k\deg P})\\
&= \prod_P (1+u^{\deg P} +  \frac{|P|^2-1}{|P|^2-2} \frac{u^{2\deg P}}{1-u^{\deg P}})\\
& = Z(u)\prod_P (1 + \frac 1{|P|^2-2} u^{2\deg P})
\end{split}
\end{equation*}
with
$$ Z(u) = \prod_P(1-u^{\deg P})^{-1} = \frac 1{1-qu}\;.
$$
We further factor
$$
\prod_P (1 + \frac 1{|P|^2-2} u^{2\deg P}) = Z(u^2/q^2) \prod_P
(1+\frac{2u^{2\deg P}-u^{4\deg P}}{|P|^2(|P|^2-2)})\;,
$$
with the product  absolutely convergent for $|u|<q^{3/4}$.

We have
\begin{equation}
\begin{split}
\sum_{j=0}^h \sum_{J\in \EM_j} \mathfrak s(J) &= \frac 1{2\pi i}
\oint F(u) \frac{1-u^{-(h+1)}}{u-1}du\\
&=\frac 1{2\pi i} \oint F(u) \frac{1}{u-1}du + \frac 1{2\pi i} \oint
F(u) \frac{u^{-(h+1)}}{1-u}du
\end{split}
\end{equation}
where the contour of integration is a small circle around the origin
not including any pole of $F (u)$, say $|u|=1/q^2$, traversed
counter-clockwise.

The first integral is zero, because the integrand is analytic near
$u=0$. As for the second integral, we shift the contour of
integration to $|u|=q^{3/4-\delta}$, and obtain
 $$
\frac 1{2\pi i} \oint F(u) \frac{ u^{-(h+1)}}{1-u}du= -\Res_{u=1/q}
- \Res_{u=1} -\Res_{u=\pm \sqrt{q}} +\frac 1{2\pi i}
\oint_{|u|=q^{3/4-\delta}} F(u) \frac{ u^{-(h+1)}}{1-u}du
$$

As $h\to \infty$, we may bound the integral around
$|u|=q^{3/4-\delta}$ by
$$
\frac 1{2\pi i} \oint_{|u|=q^{3/4-\delta}} F(u) \frac{
u^{-(h+1)}}{1-u}du \ll_q
 q^{-(3/4-\delta)(h+1)} \;,
%{\rm small}
$$
 the implied constant depending on $q$.

 The residue at $u=1/q$ gives
$$
 -\Res_{u=1/q} = \frac 1{\alpha\zeta(2)^2}  \frac{q^{h+1} }{(q-1)  }
$$
and hence its contribution to $\Var \EN$ is
\begin{equation}
\left( \frac{q^{h+1}}{\zeta(2)} \right)^2
 \end{equation}
which exactly cancels out the term  $ -\ave{\EN}^2$ in \eqref{var
nHL}.

%\fbox{This is different than the case of arithmetic progressions!}

%So we are left with the contribution of the residue at $u=\pm
%\sqrt{q}$:
%$$
%\Var = -\alpha q^{h+1}(q-1) \left\{
%\Res_{u=+\sqrt{q}}+\Res_{u=-\sqrt{q}} \right\}
%$$

The residue at $u=1$ gives
\begin{equation}
-\Res_{u=1} \frac{F(u)u^{-(h+1)}}{1-u} = F(1)  =\frac 1{1-q}
\prod_P(1+\frac 1{|P|^2-2}) = -\frac 1{(q-1)\alpha\zeta_q(2)}
\end{equation}
and its contribution to $\Var \EN$ is
\begin{equation}
-\frac{q^{h+1}}{\zeta_q(2)}
\end{equation}
which exactly cancels out the term  $ \ave{\EN }$ in \eqref{var
nHL}.

 The residue at $u=+ \sqrt{q}$ gives
 \begin{equation}
 -   \Res_{u=+\sqrt{q}} = \frac{\beta_q}{2\alpha}\frac{q^{-\frac h2-2}}{(1-\frac
 1{q^{3/2}})(1-\frac 1{q^{1/2}})}
 \end{equation}
 and the residue at $u=- \sqrt{q}$ gives
\begin{equation*}
 -   \Res_{u=-\sqrt{q}} = \frac{\beta_q}{2\alpha}(-1)^h\frac{q^{-\frac h2-2}}{(1+\frac
 1{q^{3/2}})(1+\frac 1{q^{  1/2}})}
 \end{equation*}
with $\beta_q = \prod_P(1-\frac 3{|P|^2}+\frac 2{|P|^3})$. Hence
\begin{multline*}
 -   \Res_{u=+\sqrt{q}}-\Res_{u=-\sqrt{q}} =
 %-\frac \beta \alpha
% \frac 1q \frac {1+\frac 1q}{(1-\frac 1{q^3})(q-1)}
 % \\
 + \frac {\beta_q}{\alpha} q^{-\frac h2-1}
 \frac{ (1+\frac 1{q^2}) \frac {1+(-1)^h}2  + \frac 1{q^{1/2}}
 (1+\frac 1q) \frac {1-(-1)^h}2 }{(1-\frac
 1{q^3})(q-1)}\;.
\end{multline*}
Therefore we find
%, assuming the Hardy-Littlewood heuristics are  valid, that
\begin{equation*}
%\begin{split}
\Var \EN =
 \frac{\beta_q}{1-\frac 1{q^3}} q^{(h+1)/2}
\begin{cases}
\frac {1+\frac 1{q^2}}{\sqrt{q}},& h\;{\rm even}\\    \\
\frac{1+\frac 1q}{q},&h\;{\rm odd}
\end{cases} + O(H^2nq^{-n/3}) +O_q(H^{1/4+\delta})\;.
%\\
%& = q^h (\frac 1{\zeta(2)} -\beta  \frac{ 1+\frac 1q }{1-\frac
%1{q^3}}) + \frac{\beta}{1-\frac 1{q^3}} q^{(h+1)/2}
%\begin{cases}
%\frac {1+\frac 1{q^2}}{\sqrt{q}},& h\;{\rm even}\\    \\
%\frac{1+\frac 1q}{q},&h\;{\rm odd}
%\end{cases}
%\end{split}
\end{equation*}
%Hence if $h\leq (\frac 29-\delta)n$, that is $H\ll q^{(\frac
%29-\delta)n}$, then we get an asymptotic result.
This concludes the proof of  Theorem~\ref{Thm Hall}.

 %\newpage

 \end{document}